\documentclass[12pt,oneside,english,shortlabels]{amsart}

\usepackage{babel}
\usepackage{amstext}
\usepackage{amssymb}

\usepackage{amsmath} 
\usepackage{latexsym}
\usepackage{graphicx} 

\usepackage{mathptmx}

\usepackage{graphicx} 

\usepackage{xcolor}

\makeatletter

\usepackage{amsthm}
\usepackage{enumitem}		

\usepackage{mathrsfs} 
\providecommand{\noopsort}[1]{}

\theoremstyle{definition} 

 \newtheorem{definition}{Definition}[section]
 \newtheorem{remark}[definition]{Remark}

\newtheorem*{notation}{Notations}

\theoremstyle{plain}

 \newtheorem{theorem}[definition]{Theorem}
 
 \newtheorem{lemma}[definition]{Lemma}

\makeatletter
\newcommand*{\house}[1]{
  \mathord{
    \mathpalette\@house{#1}
  }
}
\newcommand*{\@house}[2]{
  \dimen@=\fontdimen8 %
      \ifx#1\scriptscriptstyle\scriptscriptfont
      \else\ifx#1\scriptstyle\scriptfont
      \else\textfont\fi\fi
      3 %
  \sbox0{%
    $#1%
      \vrule width\dimen@\relax
      \overline{%
        \kern2\dimen@
        \begingroup 
          #2%
        \endgroup
        \kern2\dimen@
      }
      \vrule width\dimen@\relax
      \mathsurround=1.5\dimen@ 
    $
  }
  \ht0=\dimexpr\ht0-\dimen@\relax
  \dp0=\dimexpr\dp0+2\dimen@\relax
  \vbox{
    \kern\dimen@ 
    \copy0 
  }
}

\let\cal\mathcal

\def\11{{\mathbf 1}}

\def\b{\beta}

\theoremstyle{remark}
\newtheorem{exampl}[subsubsection]{Example}

\def\bee{\begin{exampl}}
\def\eee{\end{exampl}}
\def\bn{\begin{notation}}
\def\en{\end{notation}}
\def\br{\begin{remark}}
\def\er{\end{remark}}
\def\bp{\begin{prop}}
\def\ep{\end{prop}}
\def\bpr{\begin{proof}}
\def\epr{\end{proof}}
\def\bt{\begin{thm}}
\def\et{\end{thm}}
\def\be{\begin{equation}}
\def\ee{\end{equation}}
\def\bl{\begin{lem}}
\def\el{\end{lem}}
\def\bc{\begin{cor}}
\def\ec{\end{cor}}
\def\bd{\begin{defn}}
\def\ed{\end{defn}}

\author{Radhakrishnan Nair$^{\ddag}$}
\thanks{}
\address{$^{\ddag}$
Mathematical Sciences, The University of Liverpool, 
1 Peach Street, Liverpool, L69 7ZL, U.K.}
\email{nair@liv.ac.uk}
\urladdr{}

\author{Jean-Louis Verger-Gaugry$\mbox{}^{\Diamond}$ }
\thanks{}
\address{$\mbox{}^{\Diamond}$ 
LAMA, CNRS UMR 5127,
Univ. Grenoble Alpes, Univ. Savoie Mont~Blanc,
\newline
F - \!73000 Chamb\'ery, France}
\email{Jean-Louis.Verger-Gaugry@univ-smb.fr}
\urladdr{}

\author{Michel Weber$\mbox{}^{\dag}$}
\thanks{}
\address{$\mbox{}^{\dag}$
IRMA, 10  Rue du G\'en\'eral-Zimmer, 67084 Strasbourg,  Cedex, France. 
}
\email{michel.weber@math.unistra.fr}
\urladdr{}

\title[On Good Universality and the Riemann Hypothesis]
{On Good Universality and the Riemann Hypothesis}

\begin{document}

\title{On Good Universality and the Riemann Hypothesis}

\maketitle

\begin{abstract}
We use subsequence and moving average ergodic 
theorems applied to Boole's transformation and its 
variants and their invariant measures on the real 
line to give new characterisations of the 
{\color{black} Lindel\"of} Hypothesis and the 
Riemann {\color{black} Hypothesis}.  These ideas are then used 
to study the value distribution of  Dirichlet $L$ 
{\color{black} -functions}, and the zeta functions of  Dedekind, 
Hurwitz and Riemann  and their  derivatives.  This 
builds on earlier work of R. L. Adler and B. Weiss, 
M. Lifshits 
and M. Weber, J. Steuding,
J. Lee and A. I. {\color{black}Suriajaya} using 
Birkhoff's ergodic theorem  and probability theory.
\end{abstract}

\vspace{0.7cm}

\vspace{0.5cm}




\tableofcontents

\numberwithin{equation}{section} \numberwithin{figure}{section}
\newpage
\section{Introduction}
\label{S1}

Equivalent statements
of the Riemann {\color{black}Hypothesis}  are well-known
and can be found e.g. in
Titchmarsh \cite{titchmarsh} or more recently
e.g. in Mazur and Stein \cite{mazurstein}. 
The Lindel\"of Hypothesis asserts that for every
$\epsilon > 0$ we have
\begin{equation}
\label{LH}
|\zeta(1/2 + it)| = O(|t|^{\epsilon})
\qquad \quad
\mbox{as}~ |t| \to \infty,
\end{equation}
where $\zeta (z)$ is the Riemann zeta funtion,
and Landau's notation ``$O$" means
that there exists a neighbourhood $\mathcal{V}$
of $\infty$ and a constant $c \geq 0$ such that:
$|\zeta(1/2 + it)| \leq c |t|^{\epsilon}$
for all $t \in \mathcal{V}$.

Before stating the main Theorems
\ref{theorem1}--\ref{theorem1.5},
and
Theorem \ref{theorem3} in the context of
moving average ergodic theorems,
let us make precise the context in probability theory.

Let $(X_i)_{i\geq 1}$ be a sequence of 
independent Cauchy random variables, 
with characteristic function 
$\phi (t)=e^{|t|}$ and consider 
the partial sums 
$S_n= X_1+ \ldots + X_n$ $(n=1,2, \ldots )$.

M. Lifshits and M. Weber 
\cite{lifshitsweber3}
studied the 
value distribution of the Riemann zeta function
$\zeta ( s)$ sampled along the  random walk 
$(S_n)_{n\geq 1}$ showing, for $b> 2$, that
\begin{equation}
\label{limLW2}
\lim _{N\to \infty}{1\over N} \sum _{n=1}^N 
\zeta \left  ({1\over 2} +iS_n \right ) 
= 1+ o\left ({(\log N)^b \over N^{1\over 2}} 
\right ).
\end{equation}
They also showed that
\begin{equation}
\label{supIneqLW2}
\left \| {\sup_{n\geq 1}}{|\sum _{q=1}^n
\zeta({1\over 2}+iS_q)-n|\over n^{1\over 2}
( \log n)^b}\right \| _2 < \infty .
\end{equation}
Here of course $f(x) =o(g(x))$ means 
$\lim _{x\to \infty}{f(x) \over g(x)} = 0$.  
This result was extended to $L${\color{black} -functions} and Hurwitz 
zeta functions by T. Srichan 
\cite{srichan}.
In \textcolor{black}{\cite{steuding2}}
J. Steuding replaced 
$(S_n)_{n\geq 1}$ in  
\cite{lifshitsweber3}
by 
$(T^nx)_{n\geq 1}$ for almost all $x$ on 
$\mathbb{ R}$ with respect the Lebesgue measure, for 
the Boolean dynamical system by $Tx:= x-{1\over x}$.  
This result has its roots in the observation, due 
to G. Boole 
\cite{boole},
subsequently developed by 
J.W.L. Glashier 
\cite{glashier} 
\cite{glashier2},
that if $f$ 
is integrable on the real numbers, then
$$
\int _{-\infty}^{\infty} f(x)dx = 
\int _{-\infty}^{\infty} f\Bigl(x-{1\over x}\Bigr) dx.
$$
See 
\cite{adlerweiss}
for a proof of the ergodicity of 
this dynamical system.  
In Theorem 3.1 in
\cite{leeetal},
the maps
\begin{equation}
\label{Talphabeta}
T_{\alpha , \beta} (x) \ = \ 
\left\{
\begin{array}{cc}
{ \alpha \over 2} \left ( 
{x+ \beta \over \alpha } - {\alpha \over x- \beta } 
\right )
& \mbox{\quad if}~ x \ \not= \beta, \\
\beta 
& \mbox{\quad if}~ x\ = \beta,
\end{array}
\right.
\end{equation}
for $\alpha > 0$ and real $\beta$,
are shown to be measure preserving and ergodic with 
respect to the probability measure
\begin{equation}
\label{mualphabeta}
\mu _{\alpha , \beta } (A) 
= {\alpha \over \pi} 
\int _A {dt \over \alpha ^2 + (t-\beta )^2},
\end{equation}
for any Lebesgue measureable  subset $A$ of the 
real numbers.  


As noted in 
\cite{leeetal},
if $\lambda$ denotes 
Lebesgue measure then 
$\mu _{\alpha , \beta }(A) \leq {1\over \alpha \pi}
\lambda (A)$ for all $A \in \cal{B}$,
where $\mathcal{B}$ denotes the Lebesgue
$\sigma$-algebra. 
Also if 
$\phi _{\alpha , \beta} (x) = \alpha x + \beta$,  
with $T= T_{1,0}$ and $\mu = \mu _{1,0}$ 
we have
\begin{equation}
\label{TalphabetaT}
T_{\alpha , \beta} = \phi  _{\alpha , \beta} \circ 
T \circ \phi ^{-1} _{\alpha , \beta},
\end{equation}
which implies the $\mu _{\alpha , \beta }$-measure 
preservation and ergodicity of 
$T_{\alpha , \beta}$. {\color{black} The referee of this paper has brought to the authors' attention
the paper 
\cite{maugmaisrichan},
which considers a family of  
rational maps more general than that 
in 
\cite{leeetal},
which we consider in this paper.  
The suggestion is that our methods in this paper will
adapt to the more general framework in 
\cite{maugmaisrichan}
-- 
at least to a degree.  We will
not however explore this possibility in this paper, 
as it would be too much of a digression.}

\bigskip

We consider a number of general definitions.

We say $(k_n)_{n\geq 1} \in \mathbb{ N}$ is 
{\it $L^p$-good 
universal} if for each dynamical system 
$(X,\beta , \mu , T)$ and  for each 
$g \in L^p(X, \beta , \mu )$  the limit
\begin{equation}
\label{Lp_gooduniversal}
\ell _{T,g}(x)=\lim _{N\to \infty} 
{1 \over N}\sum _{n=0}^{N-1}g(T^{k_n}x)
\end{equation}
exists $\mu$ almost everywhere. 


Also throughout the rest of this paper $\cal{B}$ 
will be the Lebesgue $\sigma$- algebra on 
$\mathbb{ R}$ and when a $\sigma$-algebra on 
$\mathbb{R}$ is not explicitly mentioned it is in 
fact $\cal{B}$.

We say that 
a  sequence 
$ x_1, \ldots , x_N, \ldots $ is 
{\it uniformly 
distributed modulo one} if
\begin{equation}
\label{unifdistrmod1}
 \lim_{N\to \infty} {1 \over N} \# \{ 
 1 \leq n \leq N : \{ x_n \} \in I \} =|I| 
\end{equation}
for every interval $I \subseteq [0,1)$. For a real 
number $y$ we have used $\{ y \}$ to denote its 
fractional part and let $[y]= y-\{ y \}$ denote its 
integer part. 
We say a  sequence of integers is 
{\it uniformly distributed on $\mathbb{ Z}$}  
if it is 
uniformly distributed among the residue classes 
modulo $m$, for each natural number $m>1$.
We say a sequence of natural numbers 
$(k_n)_{n\geq 1}$ is {\it Hartman uniformly distributed 
(on $\mathbb{ Z}$)} 
if it is uniformly distributed on 
$\mathbb{ Z}$, and for each irrational number 
$\alpha$, the sequence 
$(\{ k_n \alpha \})_{n\geq 1}$ is uniformly 
distributed modulo one .  
This condition coincides 
with $(k_n)_{n\geq 1}$ being 
uniformly distributed 
on the Bohr compactification of $\mathbb{ Z}$.
Some basics on Hartman uniform distribution
can be found in the book of Kuipers and 
Niederreiter \cite{kuipersniederreiter}.
Note that if $(k_n)_{n\geq 1}$ is 
Hartman uniformly 
distributed on $\mathbb{ Z}$, 
and if for $z$ with 
$|z|=1$ we set
\begin{equation}
\label{FNz_definition}
\mbox{} \qquad 
\qquad \qquad 
F(N, z ) := {1\over N} \sum _{n=0}^{N-1} z^{k_n} 
\qquad \qquad \qquad(N=1,2,\cdots)
\end{equation}
then
\begin{equation}
\label{FNz_limite}
\lim _{N\to \infty} F(N,z) =
\left\{
\begin{array}{cc}
1 
&  \mbox{{\rm if}} \ z =1, \\
0 & 
\ {\rm otherwise.}
\end{array}
\right.
\end{equation}

\medskip

For a $c \in \mathbb{ R}$ we use $\mathbb{ H}_c$ to 
denote the half plane 
$\{ z \in \mathbb{ C} : \Re (z) > c \}$ and  use 
$\mathbb{ L}_c$  to denote
the line $\{ z \in \mathbb{ C} : \Re (z) = c \}$.  
We establish the following theorem.


\begin{theorem}
\label{theorem1} 
Let $f$ be a meromorphic  function on 
$\mathbb{ H}_c$ satisfying the following conditions:

(1)  there exists {\color{black}a} $K>0$ and a $c'>c$  
such that for any $t\in \mathbb{ R}$, we have 
\begin{equation}
\label{majoration_f}
|f(  \{ \sigma + it | \sigma > c' \} ) | \leq K;
\end{equation}

(2)  there  exists a non-increasing 
$\nu :(c,  \infty ) \to \mathbb{ R}$ such that if $\sigma$ is sufficiently near $c$ then $\nu (\sigma ) \leq
1+c- \sigma$ 
and that for any $\epsilon > 0$, 
we have $f(\sigma + it ) \ll _{f, \epsilon}
|t|^{\nu (\sigma )+\epsilon}$ as 
$|t| \to \infty $;

(3) $f$ has  at most one pole of 
order $m$ 
in $\mathbb{ H}_c$ at $s_0= \sigma_0 +it_0$, 
that is, we can write its Laurent expansion 
near $s = s_0$ as
$$
{a_{-m} \over (s-s_0)^m} + {a_{-(m-1)} \over (s-s_0)^{(m-1)}} + \ldots   {a_{-1} \over (s-s_0)^{-1 }}
+ a _0 + \sum _{n=1}^{\infty} a_n(s-s_0)^n
$$
for $m \geq 0$ where we set $m=0$ if 
$f$ has no pole in $\mathbb{ H}_c$.  

Then if 
$(k_n)_{n\geq 1}$ is $L^p$-good universal and 
Hartman uniformly distributed, for any 
$s \in \mathbb{ H}_{c} \backslash \mathbb{ L}_{\sigma_0}$,
we have
\begin{equation}
\label{lim_Lpgood_Hartman}
\lim _{N \to \infty} {1\over N} 
\sum _{n=0}^{N-1} f(s+iT^{k_n}_{\alpha , \beta }
( x)) 
= {\alpha \over \pi } \int _{\mathbb{ R}}
{f(s+ i \tau ) \over \alpha ^2 +
(\tau - \beta )^2}  d\tau ,
\end{equation}
for almost all $x$ in $\mathbb{ R}$.
\end{theorem}

In the case $k_n= n \ (n=1,2, \ldots )$ 
Theorem \ref{theorem1}
appears in 
\cite{leeetal},
where it is  
shown, 
using (1), (2) and (3)
and contour integration,  for 
$$
l_{\alpha , \beta }(s) :=  {\alpha \over \pi}   
\int _{\mathbb{ R}} {f(s+ i \tau ) \over \alpha ^2 +
(\tau - \beta )^2}  d\tau ,
$$
that if $f$ has a pole at $s_0
=
\sigma_0 + i t_0$, 
\begin{equation}
\label{l_alphabeta_classicalzeta}
l_{\alpha , \beta} (s)  
=  
\left\{
\begin{array}{cc}
{\color{black} f(s+ \alpha +i \beta)} + 
\widetilde{B_m}(s_o) , 
& \mbox{\quad if}~  c < \Re (s) < \sigma _o , 
~ s \not= 
s_o- \alpha - i \beta ,\\
\displaystyle
\sum _{n=0}^m{a_{-n} \over (-2\alpha )^n}  ,
& \mbox{\quad if}~  c < \Re (s) < \sigma _o , 
~ s = s_o- \alpha - i \beta ,\\
f(s+ \alpha + i \beta ), 
&  \mbox{\quad if}~ \Re (s) > \sigma _o ,
\end{array}
\right. 
\end{equation}
or 
$l_{\alpha, \beta}(s) = 
f(s-\alpha + i \beta)$
if $s \in \mathbb{H}_c$ and
$f$ has no pole.
Here $\Re(s)$ denotes the real part of $s$
and
$$\widetilde{B_m}(s_0) 
=
\sum  _{n=1}^m 
\Bigl\{ 
{{\color{black}a_{-n}} \over i^n(\beta + i\alpha  -i(s-s_0) )^n } 
-{ {\color{black}a_{-n}} \over i^n(\beta - i\alpha  -i(s-s_0) )^n }
\Bigr\}.$$
Moreover, when $m=1$, 
we can extend Theorem \ref{theorem1} 
to the line ${\mathbb{L}}_{\sigma_0}$ 
by setting
$$
l_{\alpha , \beta} (\sigma _0 + it) 
=
f(\sigma _0 +\alpha  +i(t+ \beta ) ) 
- { a_{-1} \alpha \over \alpha ^2 +(t_0-t-\beta )^2 }.
$$ 
Applications will be given in the next section. 

As noted earlier Steuding 
\textcolor{black}{\cite{steuding2}}
showed that if $\zeta$ denotes the Riemann zeta function

\begin{equation}
\label{limit_Steuding_zeta}
\lim _{N\to \infty} {1 \over N}\sum _{n=0}^{N-1}
\zeta (s+ iT^{n}x) 
= {1\over \pi } \int _{\mathbb{ R}} 
{\zeta (s+i \tau ) \over 1+ \tau ^2} d\tau
\end{equation}
exists $\mu$ almost everywhere
(for $x$).  We can measure the 
stability of these averages another way.
Henceforth  $C$, possibly with subscripts, 
will denote a positive constant, 
though non-necessarily the same on each occasion.

\begin{theorem}
\label{theorem4}
Suppose $s=\sigma + i t$ 
with ${1\over 2} <  \sigma < 1$.  Let

$$
Y_N (\sigma ) = Y_N (\sigma , x )
:= {1 \over N}\sum _{n=0}^{N-1}\zeta (s+ iT^{n}x).  
\eqno (N=1,2, \ldots )
$$
Suppose $(N_k)_{k \geq 1}$ is a strictly 
increasing sequence 
of positive integers.  Then there exists an absolute $C>0$ 
such that
\begin{equation}
\label{zeta_supaverages_upperbound2}
\sum _{k=1}^{\infty} \|  \sup _{N_k \leq N < N_{k+1}}|Y_N(s ) - Y_{N_k}(s )| \  \|_2^2 \leq 
{ C \over \pi \sigma}
\Bigl| {\zeta (2\sigma ) \over 2}+ {\zeta ( {2\sigma -1}) \over 2\sigma -1} 
\Bigr|.
\end{equation}
\end{theorem}

There are other ergodic theoretic measures of 
the stability of the Riemann zeta function.  
In this direction we 
consider the frequencies  associated to the 
sequences
$$
\Bigl(
{\zeta (s+ iT^{n}x)  \over n}
\Bigr)_{n \geq 1}. 
$$

Evidently
$$
\lim _{n\to \infty}{\zeta (s+ iT^{n}x)  \over n} = 0,
$$ 
for almost all $x$.  It is possible to make this more precise as follows.
Suppose $0< p < \infty $.  
For a sequence  of real numbers 
$\underline{x}=\{ x_n :  n\geq 1 \}$ 
set
\begin{equation}
\label{underline_x}
\|\underline {x} \| _{p, \infty} 
:= \bigl( 
\sup _{t>0}\, \{t^p \, \# \{n: |x_n|>t \}\}
\bigr)^{1\over p}.
\end{equation}


\begin{theorem}
\label{theorem5}  
(i) Suppose  ${1\over 2} <  \sigma < 1$.  Then
\begin{equation}
\label{limzeta_average}
\lim _{m\to \infty}{\# \{ n :{ |\zeta (\sigma + iT^nx )| \over n} \geq {1\over m } \} \over m } = \zeta ( \sigma + 1) - {2 \over \sigma (2- \sigma )},
\end{equation}
for almost all $x \in \mathbb{ R}$. and in $L^1( \mu )$ norm.

(ii) Further
\begin{equation}
\label{zeta_average_upperbound1}
\left \| \left \| \left \{ { |\zeta (\sigma + iT^nx )| \over n} : n \geq 1  \right \} \right \| _{1, \infty }  \right \| _1< \infty .
\end{equation}
\end{theorem}

Moreover,
\begin{theorem}
\label{theorem1.4}
Suppose $s=\sigma + it $ with 
${1\over 2} <  \sigma < 1$.   
Suppose $\kappa = (k_n)_{n\geq 1}$ 
is  $L^p$- good universal for $p > 2$ 
and let
\begin{equation}
\label{}
R_N(s, x , \kappa ) 
:= {1\over N} \sum _{n=0}^{N-1}\zeta (s+ iT^{k_n}x).  
\qquad \qquad (N=1,2, \ldots ).
\end{equation}
Then

\noindent
(i) there exists a 
constant $C>0$
such that
\begin{equation}
\label{}
\sum _{N=1}^{\infty} 
\bigl\|  R_{N+1}(s,.,\kappa ) - R_{N}(s,.,\kappa ) \  
\bigr\|_2^2 
\leq 
{ C \over \pi \sigma} \
\Bigl| {\zeta (2\sigma ) \over 2}+ 
{\zeta ( {2\sigma -1}) \over 2\sigma -1} 
\Bigr| ,
\end{equation}

\noindent
(ii) there exists $C>0$
\begin{equation}
\label{}
\mu _{\alpha , \beta} \Bigl( 
\Bigl\{ x : \sum _{N=1}^{\infty} 
|  R_{N+1}(s,x,\kappa ) - R_{N}(s,x,\kappa ) |^2 
\geq \lambda 
\Bigr\}  
\Bigr)
\leq 
{C \over \lambda ^2}  
\Bigl| {\zeta (2\sigma ) \over 2}
+ {\zeta ( {2\sigma -1}) \over 2\sigma -1} 
\Bigr|.
\end{equation}
\end{theorem}

We will  provide additional information  
about various special families of sequences 
in {\color{black}S}ection {\color{black}\ref{S4}} and {\color{black}S}ection \ref{S9}. 
On the other hand if we consider  
the $L^1$-norm we have the 
following theorem in the opposite direction.


\begin{theorem}
\label{theorem1.5}
Suppose $\kappa = (k_n)_{n\geq 1}$ is Hartman uniformly distributed and $L^p$- good universal for fixed 
$p  \geq \ 1$ and $\sigma \in ({1\over 2}, 1)$.  Then
\begin{equation}
\label{equation1.5}
\sum _{N\geq 1} 
\bigl|R_{N+1}(\sigma, x, \kappa )
- R_{N}(\sigma, x , \kappa )
\bigr|
 \ = \ + \infty ,
\end{equation}
almost everywhere in $x$ with respect to 
Lebesgue measure.
\end{theorem}

It would be interesting to get good bounds almost everywhere in $x$ for 
$$
\sum _{N=1} ^J 
\bigl|R_{N+1}( \sigma , x , \kappa )-
R_{N}( \sigma , x , \kappa )
\bigr|, 
\qquad \qquad (J=1,2 \ldots )
$$
even for specific cases of  $\kappa$.

\bigskip 

The paper is organized as follows: in {\color{black}S}ection 
\ref{S2} new characterisations of the 
(extended) 
Lindel\"of Hypothesis are obtained for the Riemann zeta function, and
for Dirichlet $L${\color{black} -functions}, 
Dedekind zeta functions of number fields,
Hurwitz zeta functions, 
as applications of Theorem
\ref{theorem1}. The sequences  {\color{black} here are} 
assumed to be Hartman uniformly distributed and
$L^p$-good universal. Replacing this assumption by
being Stoltz leads to similar limit
theorems, which are reported
in {\color{black}S}ection \ref{S8}.
The role played by sublinear sequences 
having controlled growth is investigated
in {\color{black}S}ection \ref{S9} for the Riemann  zeta function. 
Section \ref{S7} contains some examples of 
$L^p$-good universal sequences 
and Hartman uniformly distributed sequences.
Comparison between the dynamical and probabilistic 
models of Weber  \cite{weber}, resp. Srichan \cite{srichan}, is done
in {\color{black}S}ection \ref{S6}.

\section{Applications of Theorem \ref{theorem1}}
\label{S2}

From now, 
we assume that 
$(k_n)_{n\geq 1}$ satisfies the  hypothesis 
in Theorem \ref{theorem1}. 

In the sequel, for a function $f$ 
denote $f^{(0)} = f$ and $f^{(k)}$
the $k$-th derivative of $f$ and set
$$
P_k(s) := {(-1)^kk!\over i^{k+1}}\left ({1\over (\beta +i\alpha -i(s-1))^{k+1}} {\color{black}-} {1\over (\beta - i\alpha -i(s-1))^{k+1}} \right )
$$
for any non-negative integer $k$.

\subsection{The Riemann zeta function}
\label{S2.1}

\begin{theorem}
\label{theorem9}
Suppose $(k_n)_{n\geq 1}$ is 
$L^p$-good universal for some $p\in [1, \infty ]$ 
and Hartman uniformly distributed.
Then for any $k \geq 0$ and 
$s \in \mathbb{ H}_{-{1 \over 2}} \backslash \mathbb{ L}_1$ 
we have
\begin{equation}
\label{zetak_Lpgoodunivhartman}
\lim _{N \to \infty} {1\over N} \sum _{n=0}^{N-1} \zeta ^{(k)} (s+iT^{k_n}_{\alpha , \beta }( x)) 
= {\alpha \over \pi } \int _{\mathbb{ R}}
{\zeta ^{(k)} (s+ i \tau ) \over \alpha ^2 +(\tau - \beta )^2}  d\tau  
\end{equation}
for almost all $x$ in $\mathbb{ R}$.
\end{theorem}


Denote the right hand side of this limit by {\color{black}$l^{(k)}_{\alpha , \beta }$.  Then} 
\begin{equation}
\label{lk_alphabeta_zeta}
l^{(k)}_{\alpha , \beta } 
=  
\left\{
\begin{array}{cc}
\zeta ^{(k)}(s+ \alpha +i \beta) + P_k(s) , 
& \mbox{if}~  -{1\over 2} < \Re (s) < 1 , 
~ s \not= 1- \alpha - i \beta ,\\
(-1)^k\gamma _k- {{k!} \over (2\alpha )^{k+1}}  ,
& \mbox{if}~  -{1\over 2} < \Re (s) < 1 , ~
s = 1- \alpha - i \beta , \\
\zeta ^{(k)} (s+ \alpha + i \beta ), 
&  \mbox{if}~ \Re (s) > 1, 
\end{array}
\right. 
\end{equation}
where
\begin{equation}
\label{gammak_alphabeta_zeta}
\gamma _k = 
\lim _{N\to \infty} 
\left (  \sum _{n=1}^N {\log ^k n \over n} - 
{\log ^{k+1} N \over k+1} \right ).
\end{equation}
In the case $k=0$ 
we can extend the result 
to the line $\mathbb{ L}_1$ by setting
\begin{equation}
\label{lzero_alphabeta_zeta}
l^{(0)}_{\alpha , \beta }
=
l^{(0)}_{\alpha , \beta } (1+it) = 
\zeta ^{(0)} (1+\alpha +i(t+ \beta )) - {\alpha \over \alpha ^2 +(t^2+\beta ^2)}.
\end{equation}

If we set $k_n=n \ (n=1,2, \ldots )$ 
{\color{black}
Theorem} \ref{theorem9} 
appears in \cite{leeetal} 
and this is 
in the case $\alpha =1, \beta = 0$ 
and $k=1$ 
in 
\textcolor{black}{\cite{steuding2}}.

\textcolor{black}
{We also mention here that the case 
$k_n = n \,\,(n = 1,2, \ldots )$  
and $T=T_{{1,0}}$ of 
Theorems \ref{theorem11}
and 
\ref{theorem12}
to follow, 
appeared in 
\textcolor{black}{\cite{steuding2}},
of Theorem 
\ref{theorem12}
appeared first in 
\cite{balazardsaiasyor}
and of Theorems 
\ref{theorem14}
and 
\ref{theorem15}
appeared in 
\cite{leeetal}.}

\begin{proof}
Note that for 
$k\geq 0$, we know that $\zeta ^{(k)}$ is 
meromorphic and has an absolutely convergent 
Dirichlet {\color{black} series} for $\Re (s) > 1$.  
Thus the
condition (1) of Theorem \ref{theorem1} 
is satisfied for $c' > 1$.  
This is because the Laurent expansion of 
$\zeta$ near the pole at $s=1$ is known  
\cite{briggs} 
and then that the Laurent expansion of
$\zeta ^{(k)}$
has the form
$$
\zeta ^{(k)}(s) = 
{ (-1)^k \, k! \over (s-1)^{k+1}} + 
(-1)^k \gamma _k + 
\sum _{n=k}^{\infty }{(-1)^{n+1} \gamma _n \over (n-k+1)!} (s-1)^{n-k+1}.
$$
Thus if $k\geq 0$ the function 
$\zeta ^{(k)}$ has a pole of order $k+1$ at $s=1$.  
We can in addition show  (Titchmarsh
\cite{titchmarsh}, 
pp. 95--96) 
that given $\epsilon > 0$ we have
$$
\zeta ^{(k)}(\sigma + it)  \ll  |t|^{\mu (\sigma ) + \epsilon} ,
$$
where

$$
\mu  (\sigma )
=  
\left\{
\begin{array}{cc}
0 , 
& \mbox{if}~  \sigma > 1 ,\\
{(1-\sigma ) \over 2} ,
& \quad \quad \mbox{if}~  0 \leq  \sigma \leq 1 , \\
{1\over 2}- \sigma, 
&  \mbox{if}~ \sigma <0, 
\end{array}
\right. 
$$
if $k\geq 0$.  
Therefore Theorem \ref{theorem9} follows from
Theorem \ref{theorem1} 
with $c= -{1\over 2}$, $s_0=1$ and 
$m=k+1$ to $\zeta ^{(k)}(s)$ as required. 
\end{proof}


We now state a  formulation of the 
{\color{black} Lindel\"of} {\color{black} Hypothesis} in terms of 
$\zeta ^{(k)}$ from \cite{leeetal}. 


\begin{lemma}
\label{lemma10}
The {\color{black} Lindel\"of} {\color{black} Hypothesis} 
\eqref{LH}
is equivalent, {\color{black}for every $\epsilon > 0$, to}
\begin{equation}
\label{LH_k}
{\color{black} \left |
\zeta ^{(k)} ({1\over 2}+it)
\right | }
\ll |t|^{\epsilon}
\end{equation}
as $|t|$ tends to $\infty$, for any $k \geq 0$.
\end{lemma}

 We now give a new characterization of the {\color{black} Lindel\"of} {\color{black} Hypothesis} for the Riemann zeta function, which generalizes the classical one. The classical case $k_n = n \, 
 (n= 0, 1, 2, \ldots)$
 appears in \cite{leeetal}.

\begin{theorem}
\label{theorem11}
Suppose 
$(k_n)_{n\geq 1}$ is $L^p$-good universal 
for some $p\in [1, \infty ]$ and 
Hartman uniformly distributed.  
Suppose $k$ is any non-negative integer.  
Then the statement, for any natural number 
$l$, 
\begin{equation}
\label{limitzetak_Talphabeta_iterates}
\lim _{N \to \infty} {1\over N} \sum _{n=0}^{N-1} {\color{black} \left | \zeta ^{(k)} \left  ({1\over 2}+iT^{k_n}_{\alpha , \beta }( x) \right ) \right  |}^{\color{black}2l}
= 
{\alpha \over \pi } \int _{\mathbb{ R}}
{|\zeta ^{(k)} ({\color{black}{1\over 2}}+ i \tau )|^{\color{black}2l}\over \alpha ^2 +(\tau - \beta )^2}  d\tau 
\end{equation}
for $\mu _{\alpha , \beta }$ -almost all $x$ 
in $\mathbb{ R}$, is equivalent 
to the {\color{black} Lindel\"of} {\color{black} Hypothesis}.
\end{theorem}  

\begin{proof}
Via Lemma \ref{lemma10}
the {\color{black} Lindel\"of}
{\color{black} Hypothesis} implies that given $\epsilon > 0$ 
we have 

$$
\left |\zeta ^{(k)}{\color{black}\left ({1\over 2}+.\right )} \right |^{\color{black}2m} \in L^p (\mathbb{ R}, {\cal B},  \mu _{\alpha , \beta })  
$$
for each pair of natural numbers $k,m$.  {\color{black}Here ${.}$ represents the variable $t \in \mathbb{ R}$.}
Theorem \ref{theorem1} 
and the ergodicity of $T_{\alpha , \beta }$ 
implies that

$$
\lim _{N \to \infty} {1\over N} \sum _{n=0}^{N-1}{\color{black} \left | \zeta ^{(k)}
 \left ( {1\over 2}+iT^{k_n}_{\alpha , \beta }( x) \right ) \right |^{2l}}
= 
{\alpha \over \pi } \int _{\mathbb{ R}}
{|\zeta ^{(k)} ({\color{black}{1\over 2}}+ i \tau )|^{\color{black}2l} \over \alpha ^2 +(\tau - \beta )^2}  d\tau.
$$

{\color{black} We now prove the converse.  First}
$$
{\color{black} \zeta ^{(k)} (s) = (-1)^{k-1} \int _1^{\infty}
{[x] -x +{1\over 2} \over x^{s+1}}(\log x)^{k-1} 
(-s\log x +k)dx + {(-1)^kk! \over (s-1)^{k+1}}.}
$$
Thus there exists $C_k>0${\color{black}, dependent only on 
$k, \alpha$ and $\beta$,} such that for $|t|\geq 1$ we have
$$
|\zeta ^{(k)} ({1\over 2} + it) | <C_k|t|.
$$
Also evidently there exists $c_{\alpha , \beta } > 0$ such that if $| 
\tau | \geq 1$
$$
{1\over  \alpha ^2 + (\tau - \beta )^2} \geq c_{\alpha , \beta } {1\over 1+ \tau ^2}.
$$
{\color{black}Assuming that the Lindel\"of Hypothesis is false, implies there exists}
$\eta >0$ and $\tau _m \to \infty$  and 
$C^1_{\alpha , \beta }$ such that
$$
{\color{black} \left| \zeta ^{(m)}({1\over 2} + it)\right|} > 
C^1_{\alpha , \beta } \tau _m^{\eta}.
$$
Now{\color{black}, $|\zeta ^{(k)} ({1\over 2} + it) | <C_k|t|$} 
for any $|t| \geq 1$ with $C_k>0$, and
$$
{\color{black} \left|\zeta ^{(k)} ({1\over 2}+i\tau ) - \zeta ({1\over 2}+i\tau _m ) \right|} = \left | \int _{\tau _m}^{\tau}
|\zeta ^{(k+1)} ({1\over 2}+it )dt \right | < C^2_{\alpha , \beta }| \tau - \tau _m|  \tau.
$$
So $|\zeta ({1\over 2}+i\tau )| \geq {1\over 2}  C^1_{\alpha , \beta } \tau _m^{\eta}$ for $\tau$ with $| \tau  - \tau _m| \leq \tau _m ^{-1}$ with $m$ large enough. {\color{black}  Let $L:= {2\over 3}\tau _m$;
then the interval $I := (\tau _m - \tau _m ^{-1} , \tau _m + \tau _m  ^{-1} )$  contains the interval  $(L, 2L)$ for large $m$.}  Hence
$$
\int _L^{2L}
{\color{black} \Bigl| 
{\zeta ^{(k)}} 
\left ( {1\over 2} + i \tau \right )}
\Bigr|^{2 l} \left ({1\over 2} {d\tau \over 1+\tau ^2} \right )
\geq 
\left ( {C_1 \over 2} \right )^{2 l} \int _I \tau _m ^{2 l \eta -2}d\tau  
=
2. \left ( {C_1 \over 2} \right )^{2 l} . \tau _m ^{2 l \eta -3},
$$
which is  $\gg T^{2 l \eta -3}$, and this  is 
impossible as $l \to \infty$.  
So our theorem is proved. 
\end{proof}


We now specialize to the case $T=T_{1,0}$ 
and give a condition in terms of ergodic averages  
equivalent to the Riemann {\color{black} Hypothesis}.
\textcolor{black}{We denote by $\rho= \b +i\gamma$,  a representative complex zeros of the Riemann zeta function.
See Titschmarsh \cite{titchmarsh}, p. 30 for instance for details.}


\begin{theorem}
\label{theorem12}
Suppose $(k_n)_{n\geq 1}$ is both 
Hartman uniformly distributed and $L^p$-good universal for some $p>1$.  Then 
for almost all $x$ in $\mathbb{ R}$ with respect to Lebesgue measure we have
\begin{equation}
\label{zeta_Titerates_zeroes}
\lim _{N\to \infty} {1\over N} 
\sum _{n=1}^N \log 
{\color{black}\left | \zeta \left ({1\over 2} + 
{1\over 2}i T^{k_n}x \right )
\right |} 
=
 \sum _{\rho: \Re (\rho ) {\color{black}>} {1\over 2}}
\log \left | {\rho \over 1- \rho } \right | .
\end{equation}
\end{theorem}

Evidently, the Riemann {\color{black} Hypothesis} follows 
if either side is zero.


To prove Theorem \ref{theorem12}, 
we need the following lemma due to  M. Balazard, E. 
Saias and M. Yor \cite{balazardsaiasyor}.


\begin{lemma}
\label{lemma13}
We have
\begin{equation}
\label{logzeta_balazard_zeroes}
{1\over 2 \pi } \int _{\Re (s) 
=
 {1\over 2} } {\log |\zeta (s) | \over |s|^2} |ds| =  
\sum _{ \rho : \Re (\rho) {\color{black}>} {1\over 2}}
 \log \left | {\rho \over 1- \rho } \right |  .
\end{equation}
\end{lemma}  

\bigskip \bigskip

\begin{proof}
We wish to use 
Theorem \ref{theorem1} 
to deduce Theorem \ref{theorem12} 
using Lemma \ref{lemma13}. 
To show Theorem \ref{theorem1} 
is relevant we need to show that 
$$
\log {\color{black} \left | \zeta\left  ({1\over 2} + i. \right ) \right |} \in L^ p ({\mathbb{ R}} , 
\mu _{1,0}),
$$
{\color{black} (where $.$ denotes $t \in \mathbb{R}$)}, i.e. that
$$
\int _{\Re (s) = 
{1\over 2} }\left  ({| \log |\zeta (s))|^ p | \over |s|^2} \right ) |ds| < \infty .
$$
We mentioned earlier that there exists $C>0$ such that if $|t|>1${\color{black},}
$$
{\color{black} \Bigl |\zeta \left ({1\over 2} + i t  \right ) \Bigr |  }
\leq C|t|.
$$
Also notice that 
${(\log |\zeta (s) |)^p \over |s|^2}$ is 
continuous on an interval on 
$\Re (s) = {1\over 2}$ centred on
$s={1\over 2}$.  
Away from that interval on $\Re (s) = {1\over 2}$  
we use the observation that 
$\left |\zeta ({1\over 2} + i t) \right |  
\leq C|t|$.  
Hence {\color{black}(for $s ={1\over 2} + it$),} given $\delta > 0$ we have
$$
{(\log |\zeta (s) |)^p \over |s|^2}  
\ll 
{1\over  |s|^{2-\delta}}.
$$
This means 
$$
{\color{black} \log \left | \zeta \left  ({1\over 2} + i.  \right ) \right | }\in L^ p (\mathbb{ R} , \mu _{1,0}),
$$
as required.  Using the fact that $T$ preserves the measure $\mu _{1,0}$ and is ergodic with respect to this measure, we have
$$
\lim _{N\to \infty} {1\over N} \sum _{n=1}^N 
\log 
{\color{black} \left| \zeta \left ({1\over 2} + {1\over 2}i T^{k_n}x \right )
\right | }
= {1\over 2 \pi}  \int _{\Re(s) {\color{black} >}
{1\over 2 }} {\log |\zeta (s) | \over |s|^2} 
|ds|. 
$$
$\mu _{1,0}$ almost everywhere.  
Theorem \ref{theorem12} 
now follows from Lemma \ref{lemma13}. 
\end{proof}


\subsection{Dirichlet $L${\color{black} -functions}}
\label{S2.2}



\begin{theorem}
\label{theorem14}
Let $L(s, \chi )$ denote the $L${\color{black} -functions} 
associated to the character $\chi$. 
Suppose $(k_n)_{n\geq 1}$ is
Hartman uniformly distributed and $L^p$-good universal for some $p>1$.
Then, for $k \geq 1$,

(i) if $\chi$ is non-principal, 
for $s \in \mathbb{ H}_{-{ 1\over 2}}$
we have
$$
\lim _{N \to \infty} {1\over N} 
\sum _{n=0}^{N-1} 
L ^{(k)} (s+iT^{k_n}_{\alpha , \beta }( x), \chi ) = 
{\alpha \over \pi } \int _{\mathbb{ R}}
{L ^{(k)} (s+ i \tau, \chi  ) \over \alpha ^2 +(\tau - \beta )^2}  d\tau 
\hspace{3cm}\mbox{}
$$
\begin{equation}
\label{lseries_chinonprincipal}
\mbox{} \hspace{2cm}
=
  ~{L ^{(k)} (s +\alpha +i\beta  ,  \chi  )} 
\qquad \quad \mbox{for almost all $x$ in $\mathbb{ R}$},
\end{equation}
(ii) if $\chi (= \chi_0)$ is principal, for 
$s \in \mathbb{ H}_{-{1\over 2}}\backslash
\mathbb{ L}_{1}$ we have
\begin{equation}
\label{lseries_chiprincipal}
\lim _{N \to \infty} {1\over N} 
\sum _{n=0}^{N-1} 
L ^{(k)} (s+iT^{k_n}_{\alpha , \beta }( x), \chi ) = 
{\alpha \over \pi } \int _{\mathbb{ R}}
{L ^{(k)} (s+ i \tau, \chi  ) \over \alpha ^2 +
(\tau - \beta )^2}  d\tau ,
\end{equation}
for almost all $x$ in $\mathbb{ R}$.
\end{theorem}

Denote this limit, i.e. the
right hand side of 
\eqref{lseries_chiprincipal}, by
{\color{black}$l^{(k)}_{\alpha , \beta } (s, \chi_0)$.  Then}
\begin{equation}
\label{lk_alphabeta_dirichlet}
l^{(k)}_{\alpha , \beta } (s, \chi_0)
=  
\left\{
\begin{array}{cc}
L ^{(k)}(s+ \alpha +i \beta , \chi _0) 
+ \gamma _{-1} (\chi _0) P_k(s) , 
& \mbox{if}~  -{1\over 2} < \Re (s) < 1 , ~
s \not= 1- \alpha - i \beta ,\\
{\color{black}\gamma _k (\chi _0)} - 
{{k!} \gamma_{-1} (\chi _{0})  \over 
(2\alpha )^{k+1}}  ,
& \mbox{if}~  -{1\over 2} < \Re (s) < 1 , ~
s = 1- \alpha - i \beta , \\
 L^{(k)} (s+ \alpha + i \beta , \chi _0 ), 
 &  \mbox{if}~ \Re (s) > 1,
 \end{array}
 \right.  
\end{equation}
where $\gamma _{-1} (\chi _0), 
\gamma _k (\gamma _0)$, are constants 
that depend on $\chi _0$. 
These are  the coefficients of the Laurent 
expansion of $L^{(k)}(s, \chi _0)$ near $s=1$.  
If $k=0$ we can extend the result to the line
$\mathbb{ L}_1$ 
by defining
$$l^{(0)}_{\alpha , \beta }
=
l^{(0)}_{\alpha , \beta }(1+ it, \chi _0) 
=
L(1+ \alpha +i(t+\beta ), \chi _0) - 
{\alpha \gamma _{-1}(\chi _0) \over 
\alpha ^2 +(t^2 + \beta ^2)}.
$$


\begin{proof}
We know that $L^{(k)}{\color{black}(s, \chi )}$ has a Dirichlet {\color{black}series}
expansion for $\Re (s) >1$, 
for each non-negative integer $k$.   
From this we can show that
$$
\bigl|L^{(k)} (s, \chi )\bigr| \ll _k, _{\epsilon} 
|t|^{\mu ( \sigma) + \epsilon }
$$
where
$$
\mu  (\sigma )
=  \left\{
\begin{array}{cc}
0 , 
& \mbox{if}~  \sigma > 1 ,\\
{(1-\sigma ) \over 2} ,
& \mbox{if}~  0 \leq  \sigma \leq 1 ,\\
{1\over 2}- \sigma, 
&  \mbox{if}~ \sigma <0.
\end{array}
\right. 
$$
If $\chi$ is non-principal 
then $L^{(k)}(s, \chi )$ is entire 
for all $k\geq 0$, so $L^{(k)}(s, \chi )$ 
satisfies Theorem \ref{theorem1} 
for all $s \in \mathbb{ H}_{-{1\over 2}}$.  
If $\chi = \chi _0$ is principal, 
$L^{(k)}(s, \chi _0 ) \ (k\geq 1)$ 
has a pole of order $k+1$ at $s=1$.  
We can therefore apply  
Theorem \ref{theorem1} 
with $c= -{1\over 2}$ , $s_0 = 1$ and 
$m=k+1$, to $L^{(k)} (s, \chi _0)$ with Laurent 
coefficients coming from 
\cite{ishibashikanemitsu}. 
\end{proof}


\subsection{The Dedekind zeta function of a number field}
\label{S2.3}

We now consider the 
Dedekind $\zeta _{\mathbb{ K}}(s)$ function of a 
number field $\mathbb{ K}$ of degree 
$d_{\mathbb{ K}}$ over the rationals 
$\mathbb{ Q}$ which is defined by as follows. 
For $s \in \mathbb{C}$ such that 
$\Re (s) > 1$, let
$$
\zeta _{K}(s) = 
\sum _{I \subset \mathcal{O}_{\mathbb{K}}} 
{1\over (N_{\mathbb{K}/\mathbb{ Q}}(I))^s}.
$$  
Here $I$ runs over the ideals contained in 
the ring of 
integers of $K$ denoted ${{\mathcal{O}}_K} $ 
and $N_{K/{\mathbb{Q}}}(I)$ denotes the 
absolute norm of $I$ in $K$, which is 
the cardinality of the quotient 
${{{\cal O}_K}/ I}$.  
In the case $\mathbb{K} = {\mathbb{Q}}$, 
the function $\zeta _{\mathbb{K}} (s)$ reduces to 
the Riemann zeta function $ \zeta (s)$.  
The complex function $\zeta _{\mathbb{K}} (s)$ can 
be extended meromorphically to the 
entire complex plane with a simple pole 
at $s=1$.  See \cite{cohen}  p. 216, for more details.


\begin{theorem}
\label{theorem15}
Suppose $(k_n)_{n\geq 1}$ is 
$L^p$-good universal for some 
$p\in [1, \infty ]$ and Hartman uniformly 
distributed. For any 
$k\geq 0$ and 
any $s \in \mathbb{ H}_{{1\over 2}-{1\over d_{\mathbb{ K}}} }\backslash \mathbb{ L}_1$ 
we have
\begin{equation}
\label{limDedekindk_iteratesTalphabeta}
\lim _{N \to \infty} {1\over N} \sum _{n=0}^{N-1} \zeta ^{(k)}_{\mathbb{ K}}  
(s+iT^{k_n}_{\alpha , \beta }( x)) 
=
 {\alpha \over \pi } \int _{\mathbb{ R}}
{\zeta ^{(k)}_{\mathbb{ K}} (s+ i \tau ) \over \alpha ^2 +(\tau - \beta )^2}  d\tau  
\end{equation}
for almost all $x$ in $\mathbb{ R}$.
\end{theorem}

Denote the right hand side of this limit by 
$ l^{(k)}_{\mathbb{K}_{\alpha , \beta }}$.  
We have
\begin{equation}
\label{lk_alphabeta_dedekind}
l^{(k)}_{\mathbb{ K} _{\alpha , \beta }}(s)
=  
\left\{
\begin{array}{cc}
\zeta ^{(k)}_{\mathbb{ K}} (s+ \alpha +i \beta) 
+\gamma _{-1}({\mathbb{ K}})  P_k(s) , 
& \mbox{if}~  {1\over 2} -{1\over d_{\mathbb{ K}}} 
< \Re (s) < 1 , ~ 
s \not = 1- \alpha - i \beta ,\\
k!\gamma _k( \mathbb{ K})- {k! \gamma_{-1}(\mathbb{ K} ) \over {(2 \alpha )^{k+1}}}, 
& \mbox{if}~  {1\over 2}-{1\over d_{\mathbb{ K}}} 
< \Re (s) < 1 , ~
s = 1- \alpha - i \beta ,\\
\zeta ^{(k)} (s+ \alpha + i \beta ), 
&  \mbox{if}~ \Re (s) > 1.
\end{array}
\right.
\end{equation}
Here $\gamma _{-1}( \mathbb{ K} )$ and 
$\gamma _{k}( \mathbb{ K} )$  are constants 
{\color{black}dependent} only on $\mathbb{ K}$, 
which are coefficients of the Laurent expansion 
of $\zeta _{\mathbb{ K}}^{(k)}$ near $s=1$. 
Also  in the case $k=0$ we can extend 
the result to the line 
$\mathbb{ L}_1$ by setting
\begin{equation}
\label{lzero_alphabeta_dedekind}
l^{(0)}_{\mathbb{ K}_{\alpha , \beta }} (1+it) 
=
 \zeta ^{(0)} (1+\alpha +i(t+ \beta )) - 
 {\alpha \gamma _{-1}(\mathbb{ K}) \over \alpha ^2 +(t^2+\beta ^2)}.
\end{equation}

\begin{proof}
We refer to 
\cite{steuding} 
for a bound 
for $L$ on the half line 
and to \cite{hashimotoetal} 
for the {\color{black}coefficients} of the Laurent  
expansion $\zeta _{\mathbb{ K}}(s)$ 
near the pole $s=1$.  
We then proceed {\color{black} as earlier} with 
Theorem \ref{theorem9}
{\color{black},}  $c={1\over 2}-{1\over d_{\mathbb{ K}}}$, 
$s_0=1$ and $m=k+1$.
\end{proof}

\begin{remark} 
Results analogous to the ergodic 
characterisation of the {\color{black} Lindel\"of}
{\color{black} Hypothesis} just given,
can be proved
for 
$L${\color{black} -functions} {\color{black}associated to primitive} 
characters and Dedekind  $\zeta$ function.  
In both cases, the statements 
of ergodic characterisation are of the type:
there exist constants $k$, $\varepsilon >0 $ such that
$$
 f{\color{black} \left ({1\over 2} + it \right )} \ll _{f,k, \varepsilon  } |t|^{\epsilon},
$$
as $|t| \to \infty$, for suitable {\color{black} $L$} {\color{black} -functions}  $f$.
This inequality when $f$ is an $L$-{\color{black} function} 
\textcolor{black}{arising from Dirichlet series}
is called the 
{\color{black} Generalized} {\color{black} Lindel\"of} {\color{black} Hypothesis}.
\end{remark}


\subsection{The Hurwitz zeta function}
\label{S2.4}

Recall that the Hurwitz zeta function is 
defined for $a>0$ and $s \in {\mathbb{C}}$ 
with $\Re (s) >1 $ by  
$$\zeta (s, a) =
\sum _{n=1}^{\infty} {1\over (n+a)^s}.$$  
It is continued meromorphically to the 
whole of ${\mathbb{ C}}$ with a single pole at $s=1$.

\begin{theorem}
\label{theorem16}
Suppose $(k_n)_{n\geq 1}$ is 
$L^p$-good universal for some $p\in [1, \infty ]$ 
and Hartman uniformly distributed.
Then for any $s$ such that  $ \Re (s) > -{1\over 2}, 
s\not= 1$, with  $0 \leq a < 1$  
and $k$ a  non-negative integer,  we have
\begin{equation}
\label{limHurwitzk_iteratesTalphabeta}
\lim _{N \to \infty} 
{1\over N} \sum _{n=0}^{N-1} 
\zeta ^{(k)} (s+iT^{k_n}_{\alpha , \beta}( x), a) 
= 
{\alpha \over \pi } \int _{\mathbb{ R}}
{\zeta ^{(k)} (s+ i \tau , a ) \over \alpha ^2 +
(\tau - \beta )^2}  d\tau  ,
\end{equation}
for almost all $x$ in $\mathbb{ R}$.
\end{theorem}

\bigskip

Denote the right hand side of this limit by {\color{black}$l^{(k)}_{\alpha , \beta  }(s,a)$.  Then}
\begin{equation}
\label{lk_alphabeta_hurwitz}
l^{(k)}_{\alpha , \beta  }(s,a) 
=  
\left\{
\begin{array}{cc}
\zeta ^{(k)}(s+ \alpha +i \beta , a) + P_k(s) , 
& \mbox{if}~  -{1\over 2} < \Re (s) < 1 , ~
s \not= 1- \alpha - i \beta ,\\
(-1)^k\gamma _k(a) - {{k!} \over (2\alpha )^{k+1}}  ,
& \mbox{if}~  -{1\over 2} < \Re (s) < 1 , ~
s = 1- \alpha - i \beta , \\
\zeta ^{(k)} (s+ \alpha + i \beta ,a ), 
&  \mbox{if}~ \Re (s) > 1, 
\end{array}
\right. 
\end{equation}
where
\begin{equation}
\label{gammak_alphabeta_hurwitz}
\gamma _k(a) = 
\lim _{N\to \infty} \left (  
\sum _{n=1}^N {\log ^k (n+a) \over n+a} - 
{\log ^{k+1} (N+a) \over k+1} \right ).
\end{equation}
In the case $k=0$ 
we can extend the result to the 
line $\mathbb{ L}_1$ by setting
\begin{equation}
\label{lzero_alphabeta_hurwitz}
l^{(0)}_{\alpha , \beta } (1+it,a) 
= 
\zeta ^{(0)} (1+\alpha +i(t+ \beta ),a) - 
{\alpha \over \alpha ^2 +(t^2+\beta ^2)}.
\end{equation}
If we set $k_n=n \ (n=1,2, \ldots )$ 
Theorem \ref{theorem9} 
appears in \cite{leeetal} 
and this is the case 
$\alpha =1, \beta = 0$, and $k=1$  
in \textcolor{black}{\cite{steuding2}}.

\begin{proof}
Following the proof of Theorem 
\ref{theorem9} 
choose $c=-{1\over 2}, s_0 =1$ and $m=k+1$.  For the bound on 
$|\zeta ^{(k)}(s, a)|$ on the half line 
and the
coefficients of the Laurent series {\color{black} of}
$\zeta ^{(k)}(s,a)$ near $s=1$
we 
refer to \cite{steuding}. 
\end{proof}

\section{Proof of Theorem \ref{theorem1}}
\label{S3}

The proof is in the continuation, 
and {\color{black} a generalization},
of a Theorem of \cite{leeetal} in which {\color{black} conditions}
(1), (2) and (3) are used.
We first recall a special case of a Theorem of S. Sawyer \cite{sawyer}:

Suppose for a dynamical system 
$(\mathbb{ R} , \mathcal{B}, \mu _{\alpha , \beta }  , 
T _{\alpha , \beta})$ that 
$g \in L^p(\mathbb{ R}, \mathcal{B}, \mu _{\alpha , \beta })$ and let 
$\|g\|= (\int _X |f|^p 
d\mu _{\alpha , \beta })^{1\over p}$.  
Set
$$
Mg(x) = \sup _{N\geq 1}
\Bigl|
{1\over N} \sum _{n=1}^Ng(s+iT_{\alpha , \beta } ^{k_n}(x) ) 
\Bigr|.  \eqno (N=1,2, \dots )
$$
If $(k_n)_{n\geq 1}$ is 
$L^p$-good universal for $p > 1$, 
then there exists $C > 0$ such that 
\begin{equation}
\label{(3.1)}
\bigl\|Mg\bigr\|_p  \leq C \|g\|_p. 
\end{equation}
Because $\left  |{1\over N} 
\sum _{n=1}^Ng( s+ 
i T_{\alpha , \beta} ^{k_n}(x) ) \right | 
\leq Mg(x)$ $(N =1,2, \ldots)$ 
and $Mg \in L^p$, the dominated convergence 
theorem implies
$$
h(x) = \lim _{N\to \infty} {1\over N} 
\sum _{n=1}^Nh( s+ iT^{k_n}_{\alpha , \beta} (x) ) 
$$
exists in  $L^p$-norm.  
Our next order of business is to show 
that $h(s+ i T_{\alpha , \beta } (x)) =h(s+ ix)$.  
Let $U_{\alpha  , \beta } g(s+ i x) = 
g(s+ i T_{\alpha , \beta } ( x))$.  
This is a norm preserving operator on $L^p$ 
as $T_{\alpha , \beta }$ is $\mu _{\alpha , \beta}$ 
measure preserving.  
Also let $U^{-1}_{\alpha , \beta }$ denote the 
$L^2$ adjoint of $U_{\alpha , \beta }$.  
Recall that we say any sequence  
$(c_n)_{n \in \mathbb{ Z}}$ is positive 
definite if given a bi-sequence of complex numbers 
$(z_n)_{n\in  \mathbb{  Z}}$, 
only finitely many of whose terms are 
non-zero, we have 
$\sum _{n,m \in \mathbb{ Z}}c_{n-m}z_n 
{\overline {z_m}} \geq 0$.   
Here $\overline z$ is the conjugate of the complex 
number $z$.   
Let $\left< a,b\right > = \int _{\mathbb{ R}}  
a{ \overline b} d\mu _{\alpha , \beta }$ 
(i.e. the standard inner product on $L^2$).  
Notice that 
$\bigl(\left< U_{\alpha , \beta }^ng, 
g \right> \bigr)_{n\in \mathbb{ Z}}$ 
is positive definite.  
Recall that the Bochner-Herglotz theorem 
\cite{katznelson} 
says that there is a measure 
$\omega _g$ on ${\mathbb{ T}}$, 
such that
$$
\left< U_{\alpha , \beta } ^ng, g \right> 
= 
\int _{\mathbb{ T}}z^n d\omega _g (z). \eqno 
(n \in \mathbb{ Z})
$$ 
This tells us that
$$
\Bigl\|  {1\over N} 
\sum _{n=1}^Ng(s+ iT^{k_n+1}_{\alpha , \beta }(x) ) 
-  
{1\over N} 
\sum _{n=1}^Ng(s+ iT^{k_n}_{\alpha , \beta }(x) ) 
\Bigr\|_2
$$
$$
=\int _{\mathbb{ T}}(2-z-z^{-1}) 
\Bigl| {1\over N} 
\sum _{n=1}^Nz^{k_n} 
\Bigr| ^2d\omega _g(z)
$$
using the parametrization 
$z = e^{2\pi i \theta }$ 
for $\theta \in [0,1)$, this is
$$
=
4\int _{\mathbb{ T}} \sin ^2\left ( 
{\theta \over 2 } \right ) 
\Bigl| 
{1\over N} \sum _{n=1}^Nz^{k_n} 
\Bigr|^2 
d\omega _g (z).
$$
Using the fact that 
$\sin {\theta \over 2} = 0$ if 
$ \theta =0$ and the fact that 
$(k_n)_{n\geq 1}$ is Hartman 
uniformly distributed,
by
\eqref{FNz_definition} and \eqref{FNz_limite},
 we see that
$g(s+ iT_{\alpha , \beta }(x)) = g(x)$.
A standard fact from ergodic theory is that if 
$T_{\alpha , \beta }$ is ergodic and 
$g(s+ iT_{\alpha , \beta } (x)) =g(x)$ for 
measurable $g$ then $g(x)$ is constant, 
which must be 
$\int _{\mathbb{ R}} g d\mu _{\alpha , \beta }$.  
This extends to the $L^p$-norm  for all $p>1$.
\par
All we have to do now is to show that the 
pointwise limit is 
the same as the norm limit,  
i.e. that 
${\overline g(x)} = g(x) = 
\int _{\mathbb{  R}} gd\mu _{\alpha , \beta }$.  
We consider the sequence of natural numbers 
$(N_t)_{t\geq 1}$ such that
$$
\Bigl|{1\over N_t} 
\sum _ {n=1}^{N_t}g(s+ iT_{\alpha , \beta }^{k_n}(x)) -
\int _{\mathbb{ R}}  g(x) d \mu _{\alpha , \beta } 
\Bigr| _p \le {1\over t}.
$$
Thus
$$
\sum _{t=1}^{\infty}\int _X 
\Bigl|{1\over N_t} 
\sum _ {n=1}^{N_t}g(s+ iT_{\alpha , \beta }^{k_n}(x)) -
\int _{\mathbb{ R}}  g(x) d \mu_{\alpha , \beta } 
\Bigr|^p d\mu  < \infty .
$$
Fatou's lemma gives
$$
\int _{\mathbb{ R}} 
\Bigl( \sum _{t=1}^{\infty} 
\Bigl|{1\over N_t} 
\sum _ {n=1}^{N_t}g(s+ i T_{\alpha , \beta }^{k_n}(x)) -
\int _{\mathbb{ R}} g(x)  d \mu _{\alpha , \beta } 
\Bigr|^p 
\Bigr) d\mu  < \infty .
$$
This implies 
$$
 \sum _{t=1}^{\infty}
 \Bigl|{1\over N_t} 
 \sum _ {n=1}^{N_t}g(s+ i T_{\alpha , \beta }^{k_n}(x)) -
 \int _{\mathbb{ R}}  g(x) d \mu _{\alpha , \beta } 
 \Bigr|^p   < \infty .
$$
almost everywhere.  
This means

$$
\Bigl|{1\over N_t} 
\sum _ {n=1}^{N_t}g (s+ i T_{\alpha , \beta }^{k_n}(x)) -
\int _X g(x) d \mu _{\alpha , \beta } 
\Bigr| 
= o(1) .
$$
$\mu  _{\alpha , \beta }$ almost everywhere. 
As $(k_n)_{n\geq 1}$ is $L^p$-good 
universal we must have 
$$\lim _{N \to \infty} {1\over N} 
\sum _{n=0}^{N-1} g(s+iT^{k_n}_{\alpha , \beta } x) =
\int _{\mathbb{ R}}  g(s+ix) d\mu _{\alpha , \beta }$$ 
$\mu _{\alpha , \beta }$ almost everywhere as 
required to prove Theorem \ref{theorem1}.

\bigskip

\section{Proof of Theorems \ref{theorem4} to \ref{theorem1.5}}
\label{S4}

We begin by recalling the spectral regularization 
method of Lifshits and Weber 
\cite{lifshitsweber}
\cite{lifshitsweber2},
and
some technical preliminaries.
Assume $(X, \beta , \nu )$ is a measure space with 
finite measure $\nu$ and that 
$S$ is a $\nu$ measure preserving 
transformation of  $(X, \beta , \nu )$.  
For $f \in L^1(X, \beta , \nu )$ let
\begin{equation}
\label{A_N_definition}
B_N(f) = {1\over N}\sum _{k=0}^{N-1}f \circ T^k. 
\qquad \qquad (N=1,2, \ldots)
\end{equation}

We need the following lemma, due to R. Jones, 
R. Kaufman, J. Rosenblatt and M. Wierdl 
\cite{jonesetal}. 


\begin{lemma}
\label{lemma6}
Suppose $(N_p)_{p\geq }$ is a strictly 
increasing sequence of positive integers.  
Then there exists an absolute 
constant $C>0$ such that
\begin{equation}
\label{equation6}
\sum _{p=1}^{\infty} 
\bigl\|  
\sup _{N_p \leq N < N_{p+1}}|B_N(f ) - B_{N_p}(f)| \  \bigr\|_2^2 
\, \leq 
\,C \| f\| _2.
\end{equation}
\end{lemma}


Write $\log_{+}(u)= \max (1, \log (u ))$ 
for $u\geq 1$.  
Let $\omega$ denote 
the spectral measure associated 
to the element $f$ and define its regularised 
measure $\hat {\omega }$  
via its Radon-Nykodim 
derivative by
$$
{d \hat{\omega } \over dx}(x) = 
\int _{-\pi}^{ \pi} Q (\theta , x) 
\omega ( d \theta ),
$$
where
$$
Q( \theta , x) \ = \ 
\left\{
\begin{array}{cc}
|\theta |^{-1} \log_{+}^ 2 
(\left |{ \theta \over x } \right | ) 
& \mbox{if}~ |x| < | \theta |,\\
\theta^2 |x|^{-3} 
& \mbox{if}~ |\theta | \leq |x| \leq \pi .
\end{array}
\right.
$$
The following is a theorem of 
Lifshits and Weber 
\cite{lifshitsweber}. 


\begin{lemma}
\label{lemma7}
Suppose $N_0$ and $N_1$ with $N_0< N_1$  are positive integers.  
Then there exists an absolute 
constant $C>0$ such that
\begin{equation}
\label{equation4.2}
\bigl\|  \sup _{N_0 \leq N < N_{1}}
|B_N(f ) - B_{N_0}(f )| \  
\bigr\|_2^2 
~\leq~ 
\, C \hat {\omega} 
\bigl(
[{1\over N_1}{1\over N_0} )
\bigr).
\end{equation}
\end{lemma}

Suppose $0< p < \infty $.  
For a sequence  of real numbers 
$\underline{x}=\{ x_n :  n\geq 1 \}$ 
the quantity
$
\|\underline {x} \| _{p, \infty}$
is defined 
in \eqref{underline_x}.
Also if $r< p$ we have
$$
\| \underline {x} \| _{p, \infty} \leq \| \underline {x} \| _{p} 
\leq 
\left ( {p\over p-r} \right )^{1\over p} \| \underline {x} \| _{p, \infty}.
$$

We also consider
$$
N^*_{\underline{x}}=\sup_{m\geq 1} {\# \{ n: {x_n \over n}>{1\over m} \}\over m}.
$$

In our considerations, 
$x_n = f(\tau^n(x)) \ (n=1,2, \ldots )$  
and \,$\underline{x} =O_f(x)= 
(f(\tau ^n(x))_{n\geq 1}$  which is the value of $f$ 
along the orbit of a point $x \in X$.  
Also let $N_f(x) = N^*_{O_f(x)}$.  We have the 
following Lemma due to I. Assani \cite{assani} 
(see 
also \cite{weber} - 
we refer to Jamison, Orey and Pruitt 
\cite{jamisonoreypruitt} 
for more on the random variable case,
and \cite{raoren} for 
Birnbaum-Orlicz spaces and the ``$L \log_+ L$" notation).


\begin{lemma}
\label{lemma8}
(i) For any non-negative $f \in L^1( \mu )$ the 
function $N_f(x)$ is weak-$(p,p)$ 
for all $p \in (1, \infty)$.  Further
\begin{equation}
\label{equation4.3i}
\lim _{m\to \infty} {\# \{ n: {f(\tau ^n(x))  
\over n}>{1\over m} \}\over m} 
=
\int _X f d\mu ,
\end{equation}
$\mu$ almost everywhere and also in 
$L^1(\mu )$-norm.

(ii)  For $f \in L \log_+ L$ 

\begin{equation}
\label{equation4.3ii}
\left \| \left \| \left \{ {f(\tau ^n (x)) \over n}   \right \} \right  \| _{1, \infty} \right  \|_1 < \infty .
\end{equation}
\end{lemma}


Note that, for two positive constants $C_1 , C_2$, 
$$
C_1 N_f (x) \leq 
\left \| \left \{ {f(\tau ^n (x)) \over n} , n\geq 1\right \} \right  \|_{1, \infty}
\leq
C_2 N_f (x).
$$


We now turn to the proof of 
Theorem \ref{theorem4}. 


By Lemma \ref{lemma6}, 
$$
\sum _{k=1}^{\infty} 
\bigl\|  \sup _{N_k \leq N < N_{k+1}}
|Y_N(\sigma ) - Y_{N_k}(\sigma )| \  \bigr\|_2^2 \leq 
{C\over \pi} \int _{\mathbb{R}}
{|\zeta (\sigma + i \tau )|^2 \over 1+ \tau ^2} d \tau ,
$$
where as before $C$ is universal.  
We note that, for 
$\sigma \in ( {1\over 2} , 1 )$,
$$
 \int _{\mathbb{ R}}
{|\zeta (\sigma + i t )|^2 \over 1+ t^2} dt =  
\int _0^{\infty} {\{x\} \over  x^{2 \sigma +1}}
dx
= 
-{1\over \sigma} \left ( 
{ \zeta (2 \sigma ) \over 2} + 
{ \zeta (2 \sigma -1) \over 2 \sigma -1} \right ).
$$
Here  {\color{black} again} $\{ x \} = x- [x]$ 
denotes the fractional 
part of $x$.  
The quantity in the bracket on the 
right hand side is negative.  
Further,
$$
{1\over 2 \pi} \int _{\mathbb{ R}}{ |\zeta ({1\over 2} + it) |^2 \over {1\over 4} +t^2} dt 
= \log ( 2 \pi ) - \gamma .
$$
where $\gamma$ is Euler's constant.  
See Prop. 7,  Cor. 8 and (1.26) in Coffey 
\cite{coffey}. 
Thus
\begin{equation}
\label{CsurPi}
 {C\over \pi} \int _{\mathbb{ R}}
{|\zeta (\sigma + i \tau )|^2 \over 1+ \tau ^2} d \tau \leq   { C \over \pi \sigma}
\left | {\zeta (2\sigma ) \over 2}+ {\zeta ( {2\sigma -1}) \over 2\sigma -1} \right |.
\end{equation}
completing the proof of Theorem \ref{theorem4}.  
$\square$


We now turn to the proof of 
Theorem \ref{theorem5}. 
The first assertion is a consequence 
of Lemma \ref{lemma8} 
(i).  The limit is
$$
{1\over \pi} \int _{\mathbb{ R}}
{|\zeta (\sigma + i t )|^2 \over 1+ \tau ^2} d t .
$$
This integral is a special case of integrals 
calculated in \cite{steuding}.
As a consequence of 
Theorem 1.1 in
\cite{steuding2},
$$
{1\over \pi} \int _{\mathbb{ R}}
{|\zeta (\sigma + i t )|^2 \over 1+ \tau ^2} d t  = \zeta (\sigma +1) - {2 \over \sigma ( 2 - \sigma )}.
$$

For the second assertion, it follows from 
Lemma \ref{lemma8} 
(ii).  
The function 
$\zeta (s) \in L \log_{+} L$ 
is integrable, with respect to the 
Cauchy  measure, because of the estimate
$$
\zeta (\sigma + it ) \ll _{\epsilon} t^{{1-\sigma \over 2}+ \epsilon}.
$$
See  Titchmarsh, \cite{titchmarsh} 
section 5.1, for the details of this.  
Therefore  Theorem \ref{theorem5} 
is proved as required.  
$\square$

To prove Theorem \ref{theorem1.4} 
we need the following  two lemmas. 

Suppose $(X,\beta ,\mu )$ is a measure space and
that $T : X \to X$ is a measure preserving map.  
Let $(a_k)_{k= 0}^{\infty}$ be a
sequence of natural numbers and for any 
measurable $f$ on $X$ set
\begin{equation}
\label{CNfak_definition}
\mbox{} \qquad \qquad C_Nf(x) \ 
:= 
\ {1\over N}\sum _{k=0}^{N-1}
f(T^{a_k}x), 
\qquad \qquad (N=1,2,\cdots)
\end{equation}
that is the ergodic averages 
corresponding to the sequence and let
$$
Mf(x) \ 
:= 
\ \sup _{N\geq 1}|C_Nf(x)|.
$$
Further from the data $(X,\beta , \mu , T)$  
and $f$ we construct the
ergodic $q$-variation function
$$
V_qf(x) \ = \ (\sum _{N\geq 1}
|C_{N+1}f(x)-C_Nf(x)|^q)^{1\over q}.
\eqno (q \ \geq \ 1)
$$

Our first lemma is Theorem 1 
from 
\cite{nairweber1}.

\begin{lemma}
\label{lemma4.4}
Suppose for a sequence of natural numbers 
$(a_k)_{k=0}^{\infty}$ that for some 
$p \ > \ 1$ and 
$\widetilde{C_p} \ > \ 0$ {\color{black}depending 
only} on $p$ and $(X,\beta , \mu , T)$
we have 
\begin{equation}
\label{leqno4.1}
\bigl\|Mf \bigr\|_p \ \leq \ 
\widetilde{C_p} \|f\|_p. 
\end{equation}
Then there exists another constant $D_p \ > \ 0$ {\color{black}depending only} on $p$
and $(X,\beta , \mu , T)$ such that if $q \ > \ 1$ then
\begin{equation}
\label{leqno4.2}
\bigl\|V_qf \bigr\|_p \ \leq \ 
D_p \|f\|_p. 
\end{equation}
Further suppose there is a constant 
$\widehat{C} \ > \ 0$ {\color{black}depending only} 
on $(X,\beta, \mu , T)$ such that we have
\begin{equation}
\label{leqno4.3}
\mu (\{ x : Mf (x) \ > \ \lambda \}) \ \leq \ 
{\widehat{C} \over \lambda }
\int _X|f|\, d\mu. 
\end{equation}
Then  there is a constant $D \ > \ 0$ {\color{black}depending only} on
$(X,\beta , \mu , T)$ such that if $q \ > \ 1$ then 
\begin{equation}
\label{leqno4.4}
\mu (\{ x : V_qf (x) \ > \ \lambda \}) \ \leq \ 
{D\over \lambda }\int _X|f|\, d\mu. 
\end{equation}
\end{lemma}

\begin{proof}
Suppose $(X, {\cal L}, \eta)$ denote a 
finite measure space i.e. 
$\eta (X) < \infty$.  
Now, let $(T_n)_{n\geq 1}$ denote a sequence of 
linear transformations of 
$L^p(X, {\cal L}, \eta )$ into measurable 
functions on $X$, such that each $T_n$ is 
continuous in measure, that is, such 
that if $||f-f_m||_p $ tends to $0$ 
as $m$ tends to $\infty$, then 
$||T_n f-T_nf_m||_p $  also tends 
to $0$ as
$m$ tends to $\infty$.  
Set $T^*f(x)= \sup _{n\geq 1}|T_nf(x)|$ 
for $f \in L^p(X, {\cal L}, \eta )$.  
We will say the family
$(T_n)_{n\geq 1}$ commutes with 
$w:X \to X$ if  $T^*f(w(x)) \leq  T^*g (x)$ 
where $g(x) = f(w(x))$ on $X$.  Now suppose
$\cal F$ is a family of ergodic $\eta$ 
preserving transformations on $X$, closed under 
composition.  
We will call $(T_n)_{n\geq 1}$ distributive 
if it commutes with all the elements of 
some family $\cal F$ on $X$.   We have the 
following result of S. Sawyer 
\cite{sawyer}.

\bigskip
\begin{lemma}
\label{lemma4.5}
Let $(T_n)_{n\geq 1}$ be a distributive 
sequence of linear operators on 
$L^p(X, {\cal L}, \eta )$, where 
each $T_n$ is 
continuous in measure and 
maps  $L^p(X, {\cal L}, \eta )$ to measurable 
functions on $X$.  
If $p \in [1,2]$ and if 
$T^*f(x) < \infty$ $\eta$ almost everywhere, 
then there exists a uniform 
constant $C>0$ such that
\begin{equation}
\label{equation4.5}
\eta 
\bigl( \{ x : T^*f(x) \geq \lambda  \} 
\bigr) 
\leq {C\over \lambda ^p} \int _X|f(x)|^pd\eta ,
\end{equation}
for all $\lambda >0$ and 
$f \in L^p(X, {\cal L}, \eta )$. 
\end{lemma}

\bigskip
The conclusion of Lemma \ref{lemma4.5} 
is that if $T^*f(x) < \infty$ $\eta$ almost everywhere, then the operator $T^*$ satisfies a weak$^* (p,p)$.  
If there is a $C>0$ such that 
$||T^*f||_p \leq C||f||_p$ we say 
$T^*$ satisfies a 
{\it strong $(p,p)$ inequality}.  
It is easy to check that strong $(p,p)$ 
inequalities imply the corresponding 
weak $(p,p)$ inequalities for 
the operator $T^*$.  
On the other hand it follows 
from the Marcinkiewicz interpolation 
theorem 
\cite{steinweiss}
that  a weak $(p,p)$ inequality 
implies a strong $(q,q)$ inequality 
if $q>p$.  
We now specialize \eqref{A_N_definition}
to the situation 
where $f$ is defined on $X=\mathbb{R}$,
${\cal L}$ is the Lebesgue algebra 
on $X=\mathbb{ R}$
and $\eta = \mu _{\alpha , \beta }$:
$$
T_N f(x) 
:={1\over N} 
\sum _{n=1}^N f(T_{\alpha , \beta}^{k_n}(x)) 
\eqno (n=1,2, \dots )
$$
with $f(x) = \zeta (\sigma + ix)$.  
One checks readily that 
$(T_N)_{N\geq 1}$
commutes with $(T^n_{\alpha , \beta })_{n\geq 1}$ 
and is therefore distributive. 
In light of 
Lemma \ref{lemma4.5}, 
we see that Theorem 
\ref{theorem1.4} 
follows by 
recalling our estimate 
$$
\int _{{\mathbb{ R}}} |\zeta (\sigma + i x) |^2 
d\mu _{\alpha , \beta } 
~\leq~ 
{C\over \pi} 
\int _{{\mathbb{R}}}
{|\zeta (\sigma + i \tau )|^2 \over 1+ \tau ^2} 
d \tau 
~\leq~   { C \over \pi \sigma}
\left | {\zeta (2\sigma ) \over 2}+ 
{\zeta ( {2\sigma -1}) \over 2\sigma -1} \right |
$$
arising in \eqref{CsurPi}
in the Proof of 
Theorem \ref{theorem5}, 
Theorem \ref{theorem1.4} 
is proved. 
\end{proof}

It is possible to say more about 
specific sequences and families of sequences.  
Suppose ${\cal N} =(N_k)_{k\geq 1}$
and  
$(a_n)_{n\geq 1}$ are
sequences of natural numbers.
Consider the definition of $C_N$
in \eqref{CNfak_definition}.
Let
$$
Sf(x) 
= S( {\cal N} ,f) (x)  
:= 
\ \Bigl( \sum _{k\geq 1} 
\bigl|C_{N_{k+1}}f(x)-C_{N_k}f(x)\bigr|^2
\Bigr)^{1\over 2}.
$$ 
In the first instance, we are interested in 
conditions under which there are constants $C>0$ 
such that
\begin{equation}
\label{leqno4.5}
\bigl\|Sf \bigr\|_2 \leq C \|f\|_2. 
\end{equation}

We have the following lemma taken from 
\cite{nair2}. 

\begin{lemma}
\label{lemma4.6}
Suppose 
$1 \ < \ a \ \leq \ {N_{k+1}\over N_k}
\ \leq \ b \ < \ \infty$ and 
$a_k \ = \ \phi (n)$, 
where $\phi$ is a non-constant
polynomial mapping the natural numbers to 
themselves.  
Then there is a constant
$C \ > \ 0$ such that 
\eqref{leqno4.5} 
holds.
\end{lemma}

The following lemmas are taken from 
\cite{nairweber1}.

\begin{lemma}
\label{lemma4.7}
Suppose $1 \ < \ a \ \leq \ {N_{k+1}\over N_k}
\ \leq \ b \  < \ \infty$ and $a_k \ = 
\ \phi (p_n)$, where $\phi$ is as in
Lemma \ref{lemma4.6} 
and $p_n$ is the $n^{th}$ rational prime.
Then there is a constant
$C \ > \ 0$  such that 
\eqref{leqno4.5} 
holds.
\end{lemma}


Let
$\theta \ = \ (\theta _k )_{k\geq 1}$ 
be a ${\color{black}{\mathbb Z}}$ valued
sequence of independent identically distributed 
random variables with basic probability space 
$(\Omega , {\cal B}, P)$.  
We assume the $\sigma$-algebra
${\cal B}$ is $P$ complete and that 
there exists $\gamma \ > \ 0$ such that
${\color{black}{\mathbb E }}\bigl((\theta _1^+)^{\gamma }
\bigr) \ < \ \infty$.  
Here we have used
$\theta _1^+$ to denote $\max (\theta _1,0)$.  
Consider a strictly
increasing sequence $(q_k)_{k\geq 1}$ of natural 
numbers for which
$S$ is a bounded map from $L^2$ to itself 
when $a_k \ = \ q_k$ and
$N_k \ = \ [\rho ^k ]$ for some  $\rho \ > \ 0$, 
that is  such that
\begin{equation}
\label{leqno4.6}
\Bigl\|
\bigl(\sum _{k\geq 1} 
\bigl|
C_{N_{k+1}}(f) \ - 
\ C_{N_k}(f)
\bigr|^2(x)
\bigr)^{1\over 2}
\Bigr\|_2
\ \leq \  
C \bigl\|f \bigr\|_2. 
\end{equation}
Also assume there exists $\delta \ \in \ (0,1)$ 
such that
\begin{equation}
\label{leqno4.7}
q_k \ = \ o(2^{k^{\delta }}) 
\end{equation}
with
\begin{equation}
\label{leqno4.8}
\limsup _{k\to \infty} 
{\log \ q _{2k} \over \log \ q_k} \ 
< \ \infty 
\end{equation}
and 
$$
P(q_1 \ + \ \theta _1 \ \geq \ 0) \ = \ 1.
$$
Let
$$
C_N^{\theta}(f) \ 
:= \ {1\over N}
\sum _{k=1}^Nf(T^{q_k+\theta _k}x).
$$

\begin{lemma}
\label{lemma4.8}
If $(q_k)_{k\geq 1}$ 
satisfies \eqref{leqno4.6}, \eqref{leqno4.7}
and \eqref{leqno4.8}, 
then there is a set $\Omega _0$ contained 
in $\Omega $ of full $P$ measure
such that if $\omega \ \in \ \Omega _0$ there exists
a constant
$C>0$ such that
\begin{equation}
\label{equation4.8}
\Bigl\|
\bigl(\sum _{k\geq 1} 
\bigl|C_{N_{k+1}}^{\theta }(f) \ - \ 
C_{N_k}^{\theta }(f)
\bigr|^2(x)
\bigr)^{1\over 2}
\Bigr\|_2
\ \leq \  C\|f\|_2.
\end{equation}
\end{lemma}

Consider two strictly increasing sequences of 
natural numbers $(q_k)_{k\geq 1}$
and ${\cal N} \ = \ (N_k)_{k\geq 1}$ such that 
if $a_k \ = \ q_k$ then
\eqref{leqno4.6} 
holds.  
Let $I_q \ = \ [2^{c_q},2^{c_{q+1}})$ 
denote the $q^{th}$ interval
of the form $[2^a,2^{a+1})$ containing an element 
of ${\cal N}$.  Let
$\Phi (N) \ = \ 
(\log [q_n \ + \ q \ + 2] )^{1\over 2}$ 
if $N \ \in \ I_q$
and assume
\begin{equation}
\label{leqno4.9}
\sum _{N \in \ {\cal N}} {\Phi ^2(N) \over N} \ 
< \ \infty . 
\end{equation}
Let $\phi \ = \ (\phi _k )_{k \geq 1}$ be a 
sequence of independent, 
identically distributed random variables defined on 
a basic probability space $(\Omega ,
\beta , P)$ with $\phi _1 \ \in 
\ L^1(\Omega , \beta , P)$ such that
\begin{equation}
\label{leqno4.10}
{\color{black}{\mathbb E}}\Bigl\{ \sup _{N\in \ 
\textcolor{black}{\mathbb{N} \setminus\{0\}}
}
\bigl({\sum _{k=1}^N(\phi _k \ - 
\ {\color{black}{\mathbb E}}(\phi _k))^2 \over N}
\bigr)^{1\over 2}
\Bigr\} \ 
< \ \infty .
\end{equation}
Let
$$
w_N^{\phi}(f) \ 
:= \ {1\over N} \sum _{k=1}^N
\phi _kf(T^{q_k}x).
\eqno (N \ = \ 1,2,\cdots )
$$
\begin{lemma}
\label{lemma4.9}
Suppose 
$(q_k)_{k\geq 1}$, ${\cal N } \ 
= \
(N_k)_{k\geq 1}$ and $\phi \ 
= \ (\phi _ k)_{k \geq 1}$ satisfy 
\eqref{leqno4.8}, \eqref{leqno4.9}
and
\eqref{leqno4.10}.
Then for 
almost all $\omega$ in $\Omega$
\begin{equation}
\label{equation4.9}
\bigl\|(\sum _{k\geq 1} |w_{N_{k+1}}^{\phi _k}(f) \ - \ 
w_{N_k}^{\phi _k}(f)|^2)^{1\over 2}
\bigr\|_2
\ \leq \  C\|f\|_2.
\end{equation}
\end{lemma}
Some remarks about the nature of condition 
\eqref{leqno4.10} 
are in order.  If $N_k  \ =
\ [k^{\epsilon}]$ $(k \ = \ 1,2, \ldots )$, 
with $\epsilon \ > \ 1$ then
\eqref{leqno4.10} 
reduces to
$$
\sum _{k\geq 1}
{\log (q_{[k^{\epsilon}]}+ \log k) \over k^{\epsilon}}
\ < \ \infty,
$$
which is realised if $q_k \ = \ O(2^{k^{\delta}})$ 
for $\delta \ > \ 0$.  
Also if $N_k \ = \ 2^k$
$(k \ = \ 1,2,\ldots )$ then 
\eqref{leqno4.10}
reduces to
$$
\sum _{k\geq 1}{\log (q_{2^k}+ \log k) \over 2^k}
\ < \ \infty,
$$
which is realized if $q_k \ = \ O(2^{k^{\gamma}})$. 
Given the earlier Lemma of this section these two 
conditions are 
not difficult to {\color{black}satisfy.


Set} $k_n \ = \ 
[g(n)]$ $(n \ = \ 1,2,\ldots )$ where $g$ is a 
differentiable
function from $[0,\infty )$ to itself whose 
derivative increases with
its argument.  Let $Z_M$ denote the 
cardinality of 
the set
$\{ n \ : \ k_n \ \leq \ M \}$ and suppose 
for some function
$a \ : \ [1,\infty ) \ \to \ [1,\infty )$ 
increasing to infinity
as its argument does, that we set
$$
b(M) \ 
:= \ \sup _{{\color{black}\{\alpha \}} \in [{1\over a(M)},
{1\over 2})}
{\color{black}\left |\sum _{n:k_n \leq M} e^{2\pi ik_n \alpha } \right |.}
$$
Suppose also for some decreasing function 
$c \ : \ (1,\infty ) \ \to
\ (0,\infty )$ that
$$
{b(M) \ + \ Z_{[a(M)]} \ + \ 
{M\over a(M)} \over Z_M}
\ \leq \ C \,c(M).
$$
Then if we have 
$$
\sum _{s=1}^{\infty} c(\rho ^s) \ < \infty ,
$$
for every $\rho \ > \ 1$ we say $(k_n)_{n\geq 1}$ 
satisfies
{\em condition H}
(see examples in \textcolor{black}{S}ection \ref{S7}). 
Let $b_k \ = 
\ g^{-1}(k) \ - \ g^{-1}(k-1)$ where
$g^{-1}$ here denotes the inverse function 
of $g$ on its set of
definition.  Also suppose that 
there is a constant $C$ such that
\begin{equation}
\label{leqno4.11}
\bigl(g^{-1}({\color{black} \left [{1\over \{\alpha \}} \right ]}
)\bigr)^2
\Bigl(\sum _{{\color{black}\left  [{1\over \{\alpha \}} \right ]}N_k\geq 1}
\Bigl({1\over g^{-1}(N_k)}
\Bigr)^2
\Bigr)
\ \leq \ C. 
\end{equation}
{

\begin{lemma}
\label{lemma4.11}
Suppose $(k_n)_{n\geq 1}$ 
satisfies condition H,
and \eqref{leqno4.11} 
and that $1 \ < \ a \ 
\leq {N_{k+1}\over N_k} \ < \ b$.  
Then
\eqref{leqno4.8} 
holds for the corresponding square function.
\end{lemma}

In \cite{nair} 
it is shown that examples of sequences of 
integers which
satisfy conditions H include those given by 
$g(n) \ = \ n^{\omega}$, for
non-integer $\omega \ > 1$, $g(n) \ = 
\ e^{(\log  n)^{\gamma}}$ for
$\gamma \ \in (1, {3\over 2})$ and $g(n) \ 
= \ \alpha _k n^k \ + \ \cdots
\ + \ \alpha _1 n \ + \ \alpha _0$, where the real
numbers $\alpha _ 1, \cdots , \alpha _k$ are not 
all multiples of the 
same real number.  
The reader will readily verify 
that these examples
also satisfy conditions H and 
\eqref{leqno4.10}. 
An important point of note
is that, as shown in \cite{nair}, 
the ergodic averages, 
for $a_n \ = \
k_n$ and $f$ in $L^p$ with $p \ > \ 1$ converge 
to a $T$ invariant limit.

We want to show that if $(k_n)_{n\geq 1}$
satisfies conditions H, 
\eqref{leqno4.10} 
then \eqref{leqno4.5} 
is satisfied with
$1 \ < \ \ a \ \leq {N_{k+1}\over N_k} \ \leq \ b$.

By identifying it with its characteristic 
function  we may view a strictly
increasing sequence of natural numbers as a point 
in the power set of the
natural numbers \textcolor{black}{(i.e. $2^{\mathbb{N}}$)}, 
or as a point 
in the Cartesian product 
$\Pi _{n=1}^{\infty}X_n$ where for each natural 
number $n$, we have
set $X_n \ = \ \{ 0,1 \}$.  As 
a consequence we may put a probability measure on 
the space of strictly
increasing sequences of integers, as a product 
measure $\pi$ by setting
$\pi _n (\{1 \}) \ = \ \sigma _n$ for 
$\sigma _n \ \in \ [0,1]$ and
$\pi _n (\{ 0 \}) \ = \ 1 \ - \ \sigma _n$ 
and defining $\pi$ to be the
Cartesian product measure 
$\Pi _{n=1}^{\infty}\pi _n$.  For a strictly
increasing sequence of integers $(N_k)_{k\geq 1}$ 
suppose also that
\begin{equation}
\label{leqno4.12}
\sum _{k\geq 1} 
\Bigl({\log _2 \ N_k \over 
\sum _{n\leq N_k} \sigma _n}
\Bigr)
 \ 
 < \ \infty , 
\end{equation}
for arbitrary real $\alpha $ that we have some 
constant $C$ dependent
only on $\pi$ such that
\begin{equation}
\label{leqno4.13}
{1\over \left(\sum _{n\leq {\color{black} \left [{1\over \{\alpha \} } \right ]}}
\sigma _n
\right)^2}
\,
\displaystyle
\sum _{N_k{\color{black} \left [{1\over \{\alpha \}} \right ]}\geq 1}
\Bigl({1\over 
\sum _{n\leq N_k}\sigma _n}
\Bigr)^2 \ 
< \ C. 
\end{equation}
Then with regard to the probability measure 
just defined, we have 
the following lemmas. 

\begin{lemma}
\label{lemma4.12}
For ergodic averages with regard to almost all
strictly increasing sequences $(a_k)_{k\geq 1}$ 
with respect to $\pi$,
if $(N_k)_{k\geq 1}$ and $(\sigma _k )_{k\geq 1}$ 
satisfy 
\eqref{leqno4.12} and \eqref{leqno4.13}
above then \eqref{leqno4.8} 
holds.
\end{lemma}

\begin{lemma}
\label{lemma4.13}
Fix a natural number 
$d \ >  \ 2$ and let $(k_n)_
{n\geq 1}$ denote the set 
\begin{equation}
\label{equation4.13}
\Lambda  \ = \ \{ n \ \in \ {\mathbb{ N}}: n \ 
= \ \sum _{j\geq 1}
q_j d^j\ with \ q_j \ \in \{ 0,1 \} \},
\end{equation}
ordered by size.  
Then if 
$N_k \ = \ d^k \ (k\ = \ 1,2,\cdots )$ 
then \eqref{leqno4.5} 
holds.
\end{lemma}

For $s=\sigma + i t$ set
$$
C_N(s , \zeta, x ) 
:= {1\over N} \sum _{n=1}^N
\zeta (\sigma +iT^{k_n}(x) )
\qquad \qquad (N=1,2, \ldots ) 
$$
and set
$$
S(s, \zeta, x) := 
\ \Bigl( \sum _{k\geq 1} 
\bigl|C_{N_{k+1}}(s , \zeta, x )-
C_{N_k}(s, \zeta, x )
\bigr|^2
\Bigr)^{1\over 2}.
$$

We have the following theorem.


\begin{theorem}
\label{theorem4.14}
If $\sigma \in ({1\over 2}, 1)$ 
we have $C>0$ 
such that 
\begin{equation}
\label{equation4.14}
\bigl\|S(s, \zeta, .)
\bigr\|_2^2 
~\leq~ 
{C\over \pi \sigma } 
\left|{\zeta(2\sigma ) \over 2 \sigma } + 
{\zeta(2\sigma -1 ) \over 2\sigma - 1} 
\right|,
\end{equation}
for all the families and sequences 
$(k_n)_{n\geq 1}$ 
and ${\cal N}$ listed in 
Lemmas \ref{lemma4.6} --
\ref{lemma4.13}. 
\end{theorem}


The situation is different if we replace
$L^2$-norms by $L^1$-norms. 


\begin{lemma}
\label{lemma4.15}
Suppose $(X,\beta , \mu, T )$ is an ergodic 
and measure preserving transformation with $\mu$ 
non-ergodic.  Suppose $(k_n)_{n\geq 1}$ is Hartman 
uniform distributed and $L^p$- good universal for 
fixed $p \ \geq \ 1$.
Then for any non-constant function $f$ on 
$(X,\beta ,\mu )$ we set
$$
C_{N}( f) = \frac{1}{N}
\sum_{n = 1}^{N} f(T^{k_n}(x))
\qquad \qquad (N=1,2, \ldots ) .
$$
Then
we have 
\begin{equation}
\label{equation4.15}
\sum _{N \geq 1} 
\bigl|C_{N+1}(f)-C_{N}( f)\bigr| \ 
= \ + \infty ,
\qquad \mu ~\mbox{almost everywhere}.
\end{equation}
\end{lemma}

Applying Lemma \ref{lemma4.15} 
to 
$f(x) = \zeta ( \sigma + i x)$ for 
$\sigma \in ({1\over 2} , 1)$ 
and $ x \in \mathbb{ R}$ we get 
Theorem \ref{theorem1.5}. 



\section{Comparing dynamical and  probabilistic models}
\label{S6}

Let $Y=\Pi _{n=1}^{\infty}\Omega$, 
that is the space of sequences 
$(X_1, X_2, \cdots )$ in $\Omega$.  
Let $p_1$ denote
the projection $p_1: Y\to \Omega$ defined 
by $p_1((X_1,X_2, \cdots )) = X_1$.  
Also let $S:Y \to Y$ denote the shift map
defined on $Y$ by $S((X_1,X2, \cdots , )) 
= (X_2, X_3, \cdots )$.  
It is routine to check  that $S$ preserves the infinite
product measure 
$\mu ^{\infty} = \Pi _{n=1}^{\infty}\mu$ on $Y$.  
That the shift map $T$ is also ergodic 
with respect to this infinite product
measure is consequence of the 
{\color{black}Kolmogorov} zero one law.   
Now define $f$ on $Y$ by
 $f(\omega _1 , \omega _2 , \cdots )  
 = X_1(\omega _1 )$.  
 This means that  $X_n(\omega ) = f(S^{n-1}y) $ where
$y = (\omega _1 , \omega _2, \cdots )$ and $S$ 
denotes the above shift map.  
Also a simple computation shows that
$\int _Yf(y)d\mu ^{\infty} = {\cal E}(X_1)$.
This means that the strong law of 
large numbers follows from an 
application of Birkhoff's pointwise ergodic theorem and
the weak law of large numbers from an application of 
Von Neumann's norm ergodic theorem respectively.

The upshot of this is that, under quite 
weak hypotheses, the comparison 
between the random model described 
in \cite{weber} 
and 
the dynamical model in \cite{srichan} 
is actually a comparison 
between two different 
dynamical systems.  
Now suppose that 
$(k_n)_{n\geq 1}$ is Hartman uniformly 
distributed and $L^p$ good universal 
for fixed $p \in [1, \infty)$ 
and $\mu$ is the Cauchy distribution 
$\mu _{\alpha , \beta }$ then 
\begin{equation}
\label{(6.1)}
\lim _{N \to \infty} {1\over N} \sum _{n=0}^{N-1} f(s+iX_{k_n}( \omega )) 
= 
{\alpha \over \pi } \int _{\mathbb{ R}}
{f(s+ i \tau ) \over \alpha ^2 +
(\tau - \beta )^2}  d\tau , 
\end{equation}
for almost all $\omega$ in $\mathbb{ R}$.
We can specialise this to
\begin{equation}
\label{(6.1)bis}
\lim _{N\to \infty} {1\over N} \sum _{n=1}^N \log  {\color{black} \left | \zeta \left  ({1\over 2} 
+ {1\over 2}i X_{k_n}(\omega ) \right ) \right | } 
= 
\sum _{\rho: \Re (\rho ) {\color{black}>} {1\over 2}}
\log \Bigl| {\rho \over 1- \rho } \Bigr| .
\end{equation}

Again, the Riemann {\color{black} Hypothesis} follows if 
either side is zero.  
As above similar observations can be 
obtained other \textcolor{black}{zeta} functions and $L${\color{black} -functions}.

The condition of good universality 
is an assumption about all 
dynamical systems.  
We don't need to assume so much 
and can deduce our conclusions 
from the properties of one transformation.  
The following Theorem offers 
that link between  the two models. {\color{black}See
also 
\cite{jonesolsen}
where similar ideas appear.}


\begin{theorem}
\label{theorem_1equal2}
Consider two ergodic dynamical systems 
$(X_1, \beta _1 , \mu _1, T_1)$ and 
$(X_2, \beta _2 , \mu _2, T_2)$ both on separable measure spaces.  
Suppose that $\mu_1$ and $\mu _2$ 
are non-atomic. 
Then if for a particular sequence of 
integers $(k_n)_{n\geq 1}$ for each 
$f_1 \in L^p (X_1, \beta , \mu _1)$  
for all $p>1$ we have
\begin{equation}
\label{1equal2}
\lim _{N\to \infty}{1\over N} 
\sum _{n=1}^Nf_1(T_1^{k_n}x_1) 
=
\int _{X_1} f_1(x_1) d\mu _1 ,
\end{equation}
$\mu_1$ almost everywhere, 
then the same is true with $1$ replaced by $2$.
\end{theorem}

The condition of non-atomicity of 
$\mu_1$ and $\mu _2$ is not strictly 
necessary for the proof but it simplifies 
the proof somewhat and 
our intended applications are  to non-atomic dynamical sytems.

\begin{proof}
Let  $(X, \beta , \mu , T)$ denote a dynamical 
system  and for sequence of natural numbers 
$(k_n)_{n\geq 1}$ let
$$
m(f) = \sup _{N\geq 1} 
\Bigl| {1\over N} \sum _{n=1}^N f(T^{k_n}x)  
\Bigr|.
$$
A special case of a theorem of S. Sawyer 
\cite{sawyer}, 
tells us, after the {\color{black} hypothesis} of 
Theorem \ref{theorem_1equal2}, 
with $m = m_1$,
that 
there exists $C'_p> 0$ such that
$$
\mu \bigl( \{ x_1 \in X_1 :m_1(f_1)(x_1)  \geq \lambda \}
\bigr) \leq {C_p \| f_1\|_p \over \lambda } .
$$
Another way to say this is that 
the operator $m_1$ satisfies a 
weak-$(p,p)$ bound.  
A stronger assertion is that 
there exists $C_p*>0$
such that
\begin{equation}
\label{(6.2)}
\bigl\|m_1(f_1) \bigr\|_p \leq C_p^*\| f_1\|_p. 
\end{equation}
Here we say $m_1$ satisfies a strong-$(p,p)$ 
inequality.  
For any fixed $p$, a strong-$(p,p)$ inequality 
implies  the corresponding weak-$(p,p)$ 
inequality.  On the other hand 
as a consequence of the 
Marcinkiewicz interpolation theorem 
\cite{steinweiss}
weak-$(p_a ,p_a)$ implies strong-$(p_b, p_b)$ 
if $p_a < p_b$.  
This means that the fact that 
the weak-$(p,p)$ of $m$ for all $p>1$ 
is equivalent to it being 
strong-$(p,p)$ for all $p>1$. 

We now show that inequality 
\eqref{(6.1)} implies
\begin{equation}
\label{(6.3)}
\bigl\|m_2(f_2) \bigr\|_p 
\, \leq 
\, C_p^*\| f_2\|_p 
\end{equation}

Suppose that the dynamical systems 
$(X_1, \beta _1 , \mu _1, T_1)$ 
and $(X_2, \beta _2 , \mu _2, T_2)$ 
are both ergodic.
Now the  argument used to deduce 
Theorem \ref{theorem1} 
from \eqref{(3.1)} together with 
the ergodicy  $(X_2, \beta _2 , \mu_2 ,T_2)$
readily implies the conclusion of 
Theorem \ref{theorem9}. 

We now prove \eqref{(6.3)}.  
As both the dynamical systems 
$(X_1, \beta _1 , \mu _1, T_1)$ 
and $(X_2, \beta _2 , \mu _2, T_2)$ 
are ergodic,  
the Rokhlin-Halmos lemma says that 
given any integer $N\geq 1$ and $\epsilon > 0$, 
there exist sets $E_i \subset X_i$ $(i=1,2)$ 
with $E_i, T_i^{-1}E_i, \ldots , T_i^{-(N-1)}E_i$ 
disjoint and  
$$\mu _i (E_i \cup T_i^{-1}E_i \ldots  
\cup T_i^{-(N-1)}E_i) < 1- \epsilon .$$

Let $\gamma = {\mu _1(E_2) \over  \mu _2 (E_2)}$.  
Also for a set $D$ and a function $f$ let $f_D$ 
denote the restriction of $f$ to $D$.  

Let us observe that there always exists
a bijection from $E_1$ to $E_2$ 
that preserves the measure from $\mu_1$ 
to $\gamma \mu _2$.
This is always possible by the measure isomorphism
theorem between separable spaces.
Let us call 
$\delta$ such a 1-1 measure map 
from $E_1$ to $E_2$. It is not canonical.

Now set
$$
F_i^N= E_i \cup T_i^{-1}E_i \ldots  \cup T_i^{-(N-1)}E_i,  \eqno (i=1,2)
$$
and extend the definition of 
$ \Delta |_{E_1} : = 
\delta$ to $F_1$ by setting 
$\Delta (x)  
= 
(T_2^k \delta T_1^{-k})(x)$ for 
$x \in T^kE_1$.  
Now set
$$
Wf(y) := f(\Delta ^{-1}y)\gamma ^{1\over p}.
$$
Notice ${\rm supp} \  Wf(y) \subset T_2^kE$ 
if ${\rm supp} \ f(y) \subset T_1^kE$.  
Direct computation now gives for 
${\rm supp} \ f(y) \subset T_1^kE$
$$
\int _{T^kE_1} |f|^pd\mu_1 = 
\int _{T^kE_2} \bigl|Wf \bigr|^pd\mu_2,
$$   
and so
\begin{equation}
\label{F1f_F2Sf}
\int _{F_1} |f|^pd\mu_1 = 
\int _{F_2} \bigl|Wf \bigr|^pd\mu_2.
\end{equation}

Let
$$
C_{i,l} (f)(x) = 
{1\over l}\sum _{n=1}^lf(T_i^{k_n}x), \eqno (l=1, 2, \ldots )
$$
and set
$$
m_{i,N} (f)(x) = \sup _{1\leq l   \leq N} 
| C_{i,l} (f)(x) |,  \eqno (N=1, 2, \ldots )
$$
and evidently
$$
m_i(f)(x) = \lim _{N\to \infty}m_{i,N}(f)(x).
$$
In proving \eqref{(6.3)} 
by splitting $f$ into its real 
and imaginary parts and each of those into their 
positive and non-negative parts 
we may assume $f\geq 0$.  
Suppose $x \in T_1^kE_1$ and for 
the transformation $T_1$ we have 
$m_N(f)(x)\geq 0$.  
Then there exists $l \in [1, N]$ such that
$$
m_{i, l}(f)(x) = C_{i,l} (f)(x).
$$
Let $y=Wx$.  
Then a direct computation shows that for the 
transformation $T_2$ we have
$$
C_{2,l}(Wf)(y) =W(m_{1,N}(f))(x).
$$
From this we deduce that
$$
m_{2,N}(Wf)(y) \geq m_{1, l}(f)(x).
$$
Thus
$$
\bigl\| m_{2,N}(f) \bigr\|_p^p = 
\int _{X_2} \bigl| m_{2,N}(f)(x) \bigr|^p d\mu_2
$$ 
which is
$$
\leq \int _{F_2} \bigl| m_{2,N}(f)(x) \bigr|^pd\mu_2 
+ \epsilon .
$$ 
and we know $W$ fixes measure on 
$F_2$ and hence $L^p$ norms so we have
$$
\int _{F_2} \bigl| W(m_{2,N}(f))(y) \bigr|^pd\mu_2 
+ \epsilon 
= \bigl\| (m_{2,N}(W(f)))(y) \bigr\|_p^pd\mu_2 
+ \epsilon 
$$
$$
\leq \, c \bigl\| W(f) \bigr\|_p^p+ \epsilon ~
\leq  \,c \| f \|_p^pd + \epsilon .
$$
Now let $N \to \infty$ and 
let $\epsilon \to 0$ and 
the proof of the theorem is complete.  
\end{proof}

The ergodicity of the random dynamical system 
is implied by the {\color{black} Kolmogorov} $0-1$ law. 
If we choose $ \mu = \mu _{\alpha , \beta}$, 
we see that Theorem \ref{theorem1} 
is equivalent to \eqref{(6.1)}.  
From this we deduce all the applications of 
Theorems \ref{theorem_1equal2} 
with $T^{k_n}$ replaced by $X_{k_n}$.

\section{Hartman uniformly distributed and good universal sequences}
\label{S7}

In this section we give  some examples of 
$L^p$-good universal sequences 
for some $p\geq~\!\!\!1$.  
The examples 1, {\color{black}3--6} are Hartman uniformly 
distributed.  Example 2 is not Hartman 
uniformly distributed in general.

\begin{enumerate}
\item[1.-] {\sl The natural numbers:}

The sequence $(n)_{n=1}^\infty$ is 
$L^1$-good universal. This is 
Birkhoff's pointwise ergodic theorem.


Let $\phi$ be any non-constant polynomial mapping the natural numbers to themselves.
Note that if 
$n\in \mathbb{ N}$, 
then $n^2 \not\equiv 3\bmod 4$, 
so in general the sequences 
$(\phi(n))_{n=1}^\infty$ and
$(\phi(p_n))_{n=1}^\infty$ are 
not Hartman uniformly distributed. 
We do, however, know that if 
$\gamma\in \mathbb{ R}\setminus \mathbb{ Q},$ then 
$(\phi(n)\gamma 
)_{n=1}^\infty$ and 
$(\phi(p_n)\gamma)_{n=1}^\infty$ are 
uniformly distributed modulo $1$.

\item[2.] {\sl Condition ${\rm H}$:}  

Sequences $(k_n)_{n=1}^\infty$ that are 
both $L^p$-good universal and Hartman uniformly 
distributed can be constructed as follows.  
Set $k_n=[\tau(n)]$ $(n=1,2,\dots c)$, 
where $\tau:[1,\infty)\to[1,\infty)$ 
is a differentiable function whose derivative 
increases with its argument. 
Let $\Omega_m$ denote the cardinality 
of the set $\{n\colon a_n\le m\}$, 
and suppose, for some function 
$\varphi:[1,\infty)\to[1,\infty)$ 
increasing to infinity as its argument 
does, that we set
$$
\varrho(m) = \sup_{\{z\}\in\bigl[
{1 \over \varphi(m)},{1 \over 2}\bigr)} 
\Bigl|\sum_{n\colon k_n\le m} e(zk_n)\Bigr|.
$$
Suppose also, for some decreasing 
function $\rho:[1,\infty)\to[1,\infty)$ 
and some positive constant $\omega>0$, 
that
$$
{\varrho(m)+\Omega_{[\varphi(m)]}+
{m \over \varphi(m)}}{\Omega_m} 
\le 
\omega\rho(m).
$$
Then if we have
$$
\sum_{n=1}^\infty \rho(\theta^n) 
<\infty
$$
for all $\theta>0$, we say that 
$(a_n)_{n=1}^\infty$ satisfies 
condition ${\rm H},$ see \cite{nair}. 

Sequences satisfying condition 
${\rm H}$ are known to be both 
Hartman uniformly distributed and 
$L^p$-good universal. 
Specific sequences of integers 
that satisfy condition ${\rm H}$ include 
$a_n=[\tau(n)]$ $(n=1,2,\dots c)$ where:
     
        [I.] $\tau(n)=n^\gamma$ if $\gamma>1$ and $\gamma\notin \mathbb{ N}.$

        [II.] $\tau(n)=e^{((\log n)^{\gamma})}$ 
        for $\gamma\in(1,{3 \over 2}).$
        $\quad$ 

        [III.] $\tau(n)=b_kn^k+\dots b+b_1n+b_0$ for $b_k,\dots c, b_1$ not all rational multiplies of the same real number.

        [IV.] {\sl Hardy fields:} By a Hardy field, we mean a closed subfield (under differentiation) of the ring of germs at $+\infty$ of continuous real-valued functions with addition and multiplication taken to be pointwise. Let $\cal{H}$ denote the union of all Hardy fields. Conditions for $(a_n)_{n=1}^\infty=([\psi(n)])_{n=1}^\infty,$ where $\psi\in\cal{H}$
to satisfy condition H are given by the hypotheses of Theorems 3.4, 3.5 and 3.8. in 
\cite{boshernitzanetal}. 
Note the term ergodic is used
in this paper in place of the older term Hartman uniformly distributed.

\item[3.] {\sl A random example:} 

Suppose that 
$S=(k_n)_{n=1}^\infty$ is a strictly 
increasing sequence of natural numbers. 
By identifying $S$ with its characteristic 
function $\chi_S,$ we may view it as a 
point in $\Lambda=\{0,1\}^{\mathbb{ N}}$, 
the set of maps from $\mathbb{ N}$ to $\{0,1\}$. 
We may endow $\Lambda$ with a 
probability measure by viewing it as 
a Cartesian product 
$\Lambda=\prod_{n=1}^\infty X_n$, 
where, for each natural number $n$, 
we have $X_n=\{ 0,1\}$ and 
specify the probability $\nu_n$ on 
$X_n$ by $\nu_n(\{1\})
=
\omega_n$ with 
$0\le \omega_n\le 1$ and 
$\nu_n(\{0\})=1-\omega_n$ such that 
$\lim_{n\to\infty}\omega_n n=\infty$. 
The desired probability measure on 
$\Lambda$ is the corresponding 
product measure 
$\nu=\prod_{n=1}^\infty\nu_n$. 
The underlying $\sigma$-algebra 
$\cal{A}$ is that generated 
by the cylinders
$$
\bigl\{(\Delta_n)_{n=1}^\infty\in\Lambda\colon \Delta_{n_1}
=
\alpha_{n_1},\dots c,
\Delta_{n_k}=\alpha_{n_k}
\bigr\}
 $$
for all possible choices of 
$n_1,\dots c, n_k$ and 
$\alpha_{n_1},\dots c,\alpha_{n_k}$. 
Then almost every point 
$(a_n)_{n=1}^\infty$ in $\Lambda$, 
with respect to the measure $\nu$, 
is Hartman uniformly distributed  
(see Proposition 8.2 (i) in Bourgain
\cite{bourgain}). 
Hartman uniformly distributed 
sequences are called ergodic sequences 
in \cite{bourgain}.

\item[4.] {\sl Block sequences:} 

Suppose that $(a_n)_{n=1}^\infty=
\bigcup_{n=1}^\infty[d_n,e_n]$ is 
ordered by absolute value for disjoint 
$([d_n,e_n])_{n=1}^\infty$ with 
$d_{n-1}=O(e_n)$ as $n$ tends to infinity. 
Note that this allows the possibility that 
$(a_n)_{n=1}^\infty$ is zero density. 
This example is an immediate consequence 
of Tempelman's semigroup ergodic theorem.  
See page 218 of  \cite{bellowlosert}. 
Being a group average ergodic theorem 
this pointwise limit must be invariant, 
which ensures
that the block sequence must be 
Hartman uniformly distributed.  

\item[5.]{\sl Random perturbation of good sequences:}

Suppose that $(a_n)_{n=1}^\infty$ is an $L^p$-good 
universal sequence which is also Hartman uniformly 
distributed. 
Let $\theta=(\theta_n)_{n=1}^\infty$ 
be a sequence of $\mathbb{ N}$-valued 
independent, identically distributed 
random variables with basic probability 
space $(Y,\cal{A},\cal{P}),$ and 
a $\cal{P}$-complete $\sigma$-field 
$\cal{A}$. 
Let $ \mathbb{ E}$ denote expectation 
with respect to the basic probability 
space $(Y,\cal{A},\cal{P})$. 
Assume that there exist 
$0<\alpha<1$ and $\beta>1/\alpha$ 
such that
$$
a_n=O(e^{n^\alpha}) \mathbb{ E}\log_+^\beta|\theta_1| <\infty.
$$
Then $(k_n+\theta_n(\omega))_{n=1}^\infty$ is 
both $L^p$-good universal and 
Hartman uniformly distributed 
\cite{nairweber2}. 

\end{enumerate}

\section{Moving averages}
\label{S8}

In this paragraph we show that the assumption
of a (univariate) 
sequence
which has the property to be $L^p$ good universal and Hartman uniformly distributed
can be removed from the Theorems
\ref{theorem9},
\ref{theorem11},
\ref{theorem12},
\ref{theorem14},
\ref{theorem15}, 
\ref{theorem16},
and replaced
by the assumption of
a (bivariate) Stoltz sequence.

\subsection{Stoltz sequences and {\color{black}Theorem 7.2}}
\label{S8.1}

Let $Z$
be a collection of points in 
$\mathbb{ Z}\times \mathbb{
N}$ {\rm and let} 
$$Z^h \ := 
\ \{ (n,k) \ : \ (n,k) \
\in \ Z \ {\rm and} \ k \ \geq  \ h \},$$
$$
Z^h_{\alpha} \ := \ \{ (z,s) \ \in \ \mathbb{ Z}^2 \ 
: \ 
|z \ - \ y| \ < \ \alpha (s \ - \ r) \ {\rm for}  
\ {\rm some} \ (y,r)
\ \in Z^h \}
$$
and
$$
\mbox{}\qquad
\qquad \qquad
Z^h_{\alpha }(\lambda ) \ 
:= \ \{ n : \ 
(n, \lambda ) \ \in \
Z^h_{\alpha } \} 
\qquad \qquad \qquad( \lambda \in \mathbb{ N} ).
$$  
Geometrically we can think of $Z^1 _{ \alpha }$ as 
the lattice points contained 
in the union of all solid cones with aperture 
$\alpha$ and vertex contained in 
$Z^1= Z$.
We say a sequence of 
pairs of natural numbers 
$(n_l ,k_l )_{l=1}^{\infty}$ is \it Stoltz \rm  if 
there exists a collection of points $Z$ in 
$\mathbb{ Z}\times \mathbb{N}$, and a
function $h \ = \ h(t)$ tending to infinity with 
$t$ such
that $(n_l,k_l )_{l=t}^{\infty} \ \in \ Z^{h(t)}$ 
and
there exist $h_0$, $\alpha _0$ and $A \ > \ 0$ such
that for all integers $\lambda \ > \ 0$ we have 
$|Z^{h_0}_{\alpha _0}(\lambda)| \ \leq \ A\lambda$.  
This technical
condition is interesting because of the following 
theorem \cite{bellowjonesrosenblatt}, which will be used in the proof of Theorem
\ref{theorem3}.  \textcolor{black}{The Stoltz condition is related to the 
``cone condition" of Nagel and Stein 
\cite{nagelstein} and 
of Sueiro 
\cite{sueiro}}.

\begin{theorem}
\label{theorem2}
Let  $(X,\beta,\mu, T )$ denote a
dynamical system, with set $X$, a $\sigma$-algebra of its
subsets $\beta$, a measure $\mu$ defined on the measurable
space $(X,\beta)$ such that $\mu (X) \ = \ 1$ and a
measurable, measure preserving map $T: X \to X$.  Suppose $g$ is in $L^1(X,\beta , \mu )$ and that
the sequence of pairs on natural numbers $(n_l ,
k_l)_{l=1}^{\infty}$ is Stoltz. 
Then 
\begin{equation}
\label{limit_Stolz_measurable}
\widetilde{m_g}(x) 
:= \lim _{l\rightarrow \infty }{1\over
k_l}\sum _{j=1}^{k_l}g(T^{n_l+j}x)
\end{equation} 
exists almost everywhere with respect to $\mu$.
\end{theorem}



\begin{theorem}
\label{theorem3}
Let $f$ be a meromorphic  function on 
$\mathbb{ H}_c$ satisfying conditions (1),
(2) and (3) 
of Theorem \ref{theorem1}.  
Then if $(n_l , k_l)_{l\geq 1}$ is Stoltz, 
for any $s \in \mathbb{ H}_{c} 
\backslash \mathbb{ L}_{\sigma_0}$ ,
we have
\begin{equation}
\label{limit_Stolz_meromorphic}
\lim _{l\rightarrow \infty }{1\over
k_l}\sum _{j=1}^{k_l}  
f(s+iT^{n_l +
j}_{\alpha , \beta } (x)) 
= {\alpha \over \pi } \int _{\mathbb{ R}}
{f(s+ i \tau ) \over \alpha ^2 +
(\tau - \beta )^2} \, d\tau 
\end{equation}
for almost all $x$ in $\mathbb{ R}$.
\end{theorem} 

\begin{proof}
Let
\begin{equation}
\label{mlf_definition}
\mbox{} \qquad 
\widetilde{m_{l,f}}(x) 
:= {1\over
k_l}\sum _{j=1}^{k_l}  f(s+iT^{n_l + j}_{\alpha , \beta } (x)) 
\qquad (l=1,2, \ldots )
\end{equation}
and let
\begin{equation}
\label{mf_definition}
\widetilde{m_f}(x) 
:= \lim _{l\to \infty}{1\over
k_l}\sum _{j=1}^{k_l}  f(s+iT^{n_l + j}_{\alpha , \beta } (x)).
\end{equation}
Notice that
$$
\widetilde{m_{l,f}}(s+iT_{\alpha ,\beta} (x)) 
- \widetilde{m_{l,f}}(x) 
= 
{1\over k_l}( f(s+iT_{\alpha , \beta}^{n_l+k_l+1}(x)) - 
f(s+iT_{\alpha , \beta}^{n_l+1}(x))).
$$
This means that 
$\widetilde{m_{f}}(s+iT_{\alpha , \beta }(x)) 
= 
\widetilde{m_{f}}(s+i (x))$ $\mu _{\alpha , \beta}$ 
almost everywhere.  
As an ergodic dynamical system constant along 
almost all orbits must be constant, we must have $\widetilde{m_f}(x) = 
\int _{\mathbb{ R}} f(s+ix)d 
\mu _{\alpha , \beta }(x)$.  
We have therefore shown that
$$ 
\lim _{l\rightarrow \infty }{1\over
k_l}\sum _{j=1}^{k_l}  f(s+iT^{n_l
+j}_{\alpha , \beta } (x)) 
= 
{\alpha \over \pi } \int _{\mathbb{ R}}
{f(s+ i \tau ) \over \alpha ^2 +
(\tau - \beta )^2}  d\tau ,
$$
for almost all $x$ in $\mathbb{ R}$, as required.
\end{proof}

\subsection{Applications and moving average ergodic Theorems}
\label{S8.2}

We will forgo the proof of the 
Theorems below as they follow from 
those of Theorems 
\ref{theorem11},
\ref{theorem12},
\ref{theorem14}, 
\ref{theorem15}, 
\ref{theorem16} and 
\ref{theorem_1equal2} 
respectively via minor modification, in a  
similar way, using Theorem
\ref{theorem2} and Theorem \ref{theorem3}.

\begin{theorem}
\label{theorem11stoltz}
Suppose 
$(n_q , k_q)_{q\geq 1}$ is Stoltz.  
Suppose $k$ is any non-negative integer.  
Then the statement, for any natural number 
$l$, 
\begin{equation}
\label{limitzetak_Talphabeta_iterates}
\lim _{q \to \infty} {1\over k_q} 
\sum _{j=1}^{k_q}
\bigl| 
\zeta ^{(k)} (s+iT^{n_q + j}_{\alpha , \beta }( x)) 
\bigr|^l
= 
{\alpha \over \pi } \int _{\mathbb{ R}}
{|\zeta ^{(k)} (s+ i \tau )|^l \over \alpha ^2 +(\tau - \beta )^2} \,  d\tau 
\end{equation}
for $\mu _{\alpha , \beta }$ -almost all $x$ 
in $\mathbb{ R}$, is equivalent 
to the {\color{black} Lindel\"of}  {\color{black} Hypothesis}.
\end{theorem}

\begin{theorem}
\label{theorem12stoltz}
Suppose $(n_q , k_q)_{q \geq 1}$ is Stoltz.  
Then 
for almost all $x$ in $\mathbb{ R}$ 
with respect to Lebesgue measure we have
\begin{equation}
\label{zeta_Titerates_zeroes}
\lim _{q \to \infty} {1\over k_q} 
\sum _{j=1}^{k_q} \log 
\bigl| \zeta ({1\over 2} + 
{1\over 2}i T^{n_q + j}x)
\bigr| 
=
 \sum _{\rho: \Re (\rho ) {\color{black}>} {1\over 2}}
\log 
\Bigl| {\rho \over 1- \rho } \Bigr| .
\end{equation}
\end{theorem}
Again, if either side is zero, this 
is equivalent to the Riemann {\color{black} Hypothesis}.

We now consider Dirichlet $L$-{\color{black}functions}.

\begin{theorem}
\label{theorem14stoltz}
Let $L(s, \chi )$ denote the $L$-series 
associated to the character $\chi$.  
Suppose $(n_q , k_q)_{q \geq 1}$ is Stoltz,
and let $k$ be a nonnegative integer.
Then,

(i) if $\chi$ is non-principal, 
for $s \in \mathbb{ H}_{-{ 1\over 2}}\backslash
\mathbb{ L}_{1}$ we have
$$
\lim _{q \to \infty} {1\over k_q} 
\sum _{q=1}^{k_q} 
L ^{(k)} (s+iT^{n_q + j}_{\alpha , \beta }( x), \chi ) = 
{\alpha \over \pi } \int _{\mathbb{ R}}
{L ^{(k)} (s+ i \tau, \chi  ) \over \alpha ^2 +(\tau - \beta )^2} \, d\tau 
\hspace{3cm}\mbox{}
$$
\begin{equation}
\label{lseries_chinonprincipalstoltz}
\mbox{} \hspace{1cm}
=
  ~{L ^{(k)} (s +\alpha +i\beta  ,  \chi  )} 
\qquad \mbox{for almost all $x$ in $\mathbb{ R}$},
\end{equation}
(ii) if $\chi$ is principal, for 
$s \in \mathbb{ H}_{-{1\over 2}}\backslash
\mathbb{ L}_{1}$ we have
\begin{equation}
\label{lseries_chiprincipalstoltz}
\lim _{q \to \infty} {1\over k_q} 
\sum _{j=1}^{k_q} 
L ^{(k)} (s+iT^{n_q + j}_{\alpha , \beta }( x), \chi ) = 
{\alpha \over \pi } \int _{\mathbb{ R}}
{L ^{(k)} (s+ i \tau, \chi  ) \over \alpha ^2 +
(\tau - \beta )^2} \, d\tau ,
\end{equation}
for almost all $x$ in $\mathbb{ R}$.
\end{theorem}




We now  consider the Hurwitz zeta function.

\begin{theorem}
\label{theorem16_stoltz}
Suppose $(n_q , k_q)_{q \geq 1}$ is 
Stoltz. 
For any $s$ such that  $ \Re (s) > -{1\over 2}, 
s \not= 1$, $0 \leq a < 1$  
and $k$ a non-negative integer,  we have
\begin{equation}
\label{limHurwitzk_iteratesTalphabetastoltz}
\lim _{q \to \infty} {1\over k_q} 
\sum _{j=1}^{k_q} 
\zeta ^{(k)} (s+iT^{n_q + j}_{\alpha , \beta}( x), a) 
= 
{\alpha \over \pi } \int _{\mathbb{ R}}
{\zeta ^{(k)} (s+ i \tau , a ) \over \alpha ^2 +(\tau - \beta )^2} \, d\tau  ,
\end{equation}
for almost all $x$ in $\mathbb{ R}$.
\end{theorem}

In \eqref{limHurwitzk_iteratesTalphabetastoltz}
the right hand side is given by
\eqref{lk_alphabeta_hurwitz},
\eqref{gammak_alphabeta_hurwitz}
and 
\eqref{lzero_alphabeta_hurwitz}.

The analogue of Theorem
\ref{theorem_1equal2}
is now the following.

\begin{theorem}
\label{theorem_1equal2stoltz}
Consider two ergodic dynamical systems 
$(X_1, \beta _1 , \mu _1, T_1)$ and 
$(X_2, \beta _2 , \mu _2, T_2)$,
both on separable measure spaces.  
Suppose that $\mu_1$ and $\mu _2$ 
are non-atomic.
Then if for a Stoltz sequence of integers
$(n_q , k_q)_{q \geq 1}$
and each
$f_1 \in L^{p}(X_1 , \beta_1 , \mu_1)$
 we have
\begin{equation}
\label{1equal2stoltz}
\lim _{q \to \infty}{1\over k_q} 
\sum _{j=1}^{k_q} 
f_1(T_1^{n_q + j}x_1) 
=
\int _{X_1} f_1(x_1) d\mu _1 ,
\end{equation}
$\mu_1$ almost everywhere, 
then the same is true with $1$ replaced by $2$.
\end{theorem}


\section{Sublinearity and the Riemann zeta function}
\label{S9}

In the same vein as Theorem \ref{theorem4},
Theorem \ref{theorem5} and Theorem \ref{theorem1.4},
now with peculiar sublinear sequences,
the following Theorem can be deduced.
 
The following Lemma is taken from 
\cite{lifshitsweber2}.
Let $I_A$ denote the indicator function of the set $A$.

\begin{lemma}
\label{Lemma9.1}
Suppose $k_q= k(q),  q = 1, 2, \dots $, 
with $k$ sub-linear such that $\sup _u k'(u)$ 
is absolutely bounded.  
Also assume
$$
K(\theta ) = {1\over |\theta |} 
\int _{\min ({1\over |\theta |},1)}
{k'(u) \over |u|^2}du + 
k\left ({1\over | \theta |} \right )|\theta |
\,I_{\{ |\theta | \leq 1\}},
$$  
is uniformly bounded on $\mathbb{R}$. 
Suppose $(X, \beta, \mu , T)$ is  a dynamical system  
and $f \in L^2(X, \beta , \mu )$ and set
$$
g^k_n(f(x))={1\over n} 
\sum _{j=k_n}^{k_n+n-1}f(T^{j}(x)).
$$
Then, for any strictly increasing 
sequence of natural numbers 
$(n_q)_{q\geq 1}$, there exists $C> 0$ such that
$$
\Bigl\| 
\Bigl( \sum _{q \geq 1}
| g^k_{n_{q+1}}(f) - g^k_{n_{q}}(f)|^2
\Bigr)^{1\over 2} \Bigr\|_2 
\leq C \| f \|_2.
$$\end{lemma}

We now specialize to the case, $f(x)= \zeta ( \sigma +ix)$ with $\sigma \in [{1\over 2} , 1)$.

\begin{theorem}
\label{theorem9.2}
Suppose $k$ and  $K $ are as in Lemma \ref{Lemma9.1}. 
Now 
set, with $\sigma \in [{1\over 2} , 1)$,
$$
G^{k}_{n, \sigma} (x) = G^k_n(\zeta (\sigma +ix))={1\over n} \sum _{j=k_n}^{k_n+n-1}\zeta (\sigma + iT^{j}(x)).
$$
Then, for any strictly increasing sequence of 
natural numbers $(n_q)_{q \geq 1}$, 
there exists a constant $C> 0$ such that
$$
\Bigl\| 
\Bigl( \sum _{q \geq 1}
\bigl| G^k_{n_{q+1}, \sigma }(\zeta ) - G^k_{n_{q}, \sigma }(\zeta )
\bigr|^2
\Bigr)^{1\over 2} 
\Bigr\|_2 
\leq { C \over \pi \sigma}
\left | {\zeta (2\sigma ) \over 2}+ {\zeta ( {2\sigma -1}) \over 2\sigma -1} \right |.
$$
\end{theorem}

\section*{Acknowledgements}  
We thank the referee for very detailed comments, substantially improving the presentation of this paper, and for bringing to our attention important additional references of material significance to this topic.

\end{document}